%% file: ex_article.tex
\documentclass[review,onefignum,onetabnum]{siamart171218}

\input{ex_shared}

\ifpdf
\hypersetup{
  pdftitle={On multilevel Monte Carlo methods for deterministic and uncertain hyperbolic systems},
  pdfauthor={Junpeng Hu, Shi Jin, Jinglai Li and Lei Zhang}
}
\fi

\nolinenumbers

\begin{document}

\maketitle

\begin{abstract}
In this paper, we evaluate the performance of the multilevel Monte Carlo method (MLMC) for deterministic and uncertain hyperbolic systems, where randomness is introduced either in the modeling parameters or in the approximation algorithms. MLMC is a well known variance reduction method widely used to accelerate Monte Carlo (MC) sampling. However, we demonstrate in this paper that for hyperbolic systems, whether MLMC can achieve a real boost turns out to be delicate. The computational costs of MLMC and MC depend on the interplay among the accuracy (bias) and the computational cost of the numerical method for a single sample, as well as the variances of the sampled MLMC corrections or MC solutions. We characterize three regimes for the MLMC and MC performances using those parameters, and  show that MLMC may not accelerate MC and can even have a higher cost when the variances of MC solutions and MLMC corrections are of the same order. Our studies are carried out by a few prototype hyperbolic systems: a linear scalar equation, the Euler and shallow water equations, and a linear relaxation model, the above statements are proved analytically in some cases, and demonstrated numerically for the cases of the stochastic hyperbolic equations driven by white noise parameters and  Glimm's random choice method for deterministic hyperbolic equations. 

    
\end{abstract}


\overfullrule=0pt

\section{Introduction}
The Monte Carlo (MC) method is a widely used computational tool from multidimensional integration, to statistical physics and finance. It is a sampling method which generates a large number of samples to retrieve statistical information that can  be a good approximation of  a given random or deterministic  problem, without curse-of-dimensionality. The accuracy of MC is governed by the variance of those samples and the bias induced by the discretization method. To further reduce the variance and to accelerate the sampling algorithm, Multilevel Monte Carlo (MLMC) method \cite{giles2008multilevel} can be adopted by decomposing the calculation of the quantity of interest (QoI) to hierarchical levels. MLMC has achieved great success for the efficient simulation of stochastic differential equations (SDE) \cite{abdulle2013stabilized,dereich2011multilevel,giles2008multilevel},  elliptic PDEs with random coefficients  \cite{barth2011multi,cliffe2011multilevel}, and particle methods for transport equations \cite{lovbak2021multilevel}.

In this paper, we consider linear and nonlinear hyperbolic systems  where MC and MLMC can be applied in the following circumstances: i), for hyperbolic equations with random parameters; ii),  for randomized algorithms of deterministic hyperbolic equations. These two cases have interesting applications in random algorithms for deterministic problems \cite{chorin1976, Glimm1965solutions}, or problems in uncertainty quantification \cite{AbMi, ayi2019analysis, JP-UQ, najm-UQCFD}. Although MLMC has be shown  to accelerate MC for many stochastic PDEs and randomized algorithms, we will demonstrate that, at least in the context of hyperbolic systems, the situation is more subtle. 

Let us start by considering the stochastic differential equation ${\rm d} u(t) = a(u, t) {\rm d}t + b(u,t) {\rm d}W$. To compute certain QoI, for example, $\E\left[u(T)\right]$, one can employ a discretization method such as the explicit Euler method, $u_{n+1} = u_{n} + a(u_n, t_n) \Delta t + b(u_n, t_n)\Delta W$, with time step $\Delta t$. Let $N$ be the number of samples, and $\sigma$ be the standard deviation of each sample. It is well known that, the mean square error ($\mse$) is bounded by $c_1 N^{-1}\sigma + c_2 \Delta t^2$,
where $c_1$, $c_2$ are positive constants. The first term comes from the statistical error of the MC estimation, and the second term comes from the bias, which is in turn bounded by the discretization error of the numerical method. For SDE, $\sigma$ is a constant independent of $\Delta t$, and we can take $\Delta t\approx  O(\delta)$, $N\approx O(\delta^{-2})$, such that the $\mse$ is $O(\delta^2)$ and the computational cost is $O(\delta^{-3})$. 


The multilevel Monte Carlo (MLMC) method proposed by Giles \cite{giles2008multilevel} et. al. is one of the state-of-the-art variance reduction techniques to further reduce the computational cost. MLMC samples the solution $P_0^{(0)}$ on the coarsest level and corrections $P^{(l)}_l - P^{(l-1)}_l$ on the finer levels, then assembles an estimator for the QoI on the finest mesh by a telescoping sum. A sample correction $P^{(l)}_l - P^{(l-1)}_l$ can be obtained using solutions $P^{(l)}_l$, $P^{(l-1)}_l$ with the same set of random variables (random trajectory) on the $l$th level, but with different time steps $\Delta t_l $ and $\Delta t_{l-1}$. The key to MLMC is the reduction of $\Var\left[P_l^{(l)} - P_l^{(l-1)}\right]$ without changing the expectation $\E\left[P_l^{(l)} - P_l^{(l-1)}\right]$. It can be proved that for SDE, the computational cost of MLMC can be reduced to $O(\delta^{-2}(\log\delta)^2)$ \cite[Theorem 3.1]{giles2008multilevel}, under certain conditions for the variances of corrections.


For hyperbolic equations, as one needs to consider both spatial and temporal variables, the situation is more complex. It turns out that, whether or not MLMC can enhance the performance of a MC sampling algorithm for  hyperbolic problems, depends on the interplay between the bias induced by the numerical method for a single sample, variances of the sampled MC solutions or MLMC corrections, and the computational cost of the numerical method for a single sample, which can be characterized by problem dependent parameters.

To be more specific, let $P_l^{(l)}$ be the MC solution at the $l$th level, and $Y_l:=P^{(l)}_l-P^{(l-1)}_{l}$ be the MLMC correction at the $l$th level ($l\geq 1$). Let 
\begin{itemize}
    \item $\alpha$ denote the order of accuracy for the numerical method of a single sample (bias). In this paper, we always use first order method with $\alpha=1$ unless otherwise stated. 
    \item $\beta_0$ and $\beta$ denote the orders of the variances of MC solutions and MLMC corrections, respectively. Namely $\mathsf{V}_l :=\Var\left[P_l^{(l)}\right] \leq c_2 (\Delta t_l)^{\beta_0}$, and $V_l:=\Var\left[Y_l\right] \leq c_2 (\Delta t_l)^{\beta}$.  
    \item $\gamma$ denote the order of the computational cost for a single sample such that the cost is $O(\Delta t^{-\gamma})$. For explicit numerical methods, $\gamma$ corresponds to the physical dimension of the problem, for example, $\gamma=1$ for SDE, and $\gamma=2$ for 1D SPDE. 
\end{itemize}

We can estimate the cost of MLMC and MC algorithms through parameters $\alpha$, $\beta$ or $\beta_0$, $\gamma$, and the desired accuracy $\delta>0$, using \cref{thm:costmlmc} and \cref{thm:costmc} in \S~\ref{sec:mlmc}. In general, when $\beta_0\leq\beta \leq  \gamma$, up to a logarithmic factor, the cost of MC is $O(\delta^{-2-(\gamma-\beta_0)/\alpha})$, and the cost of MLMC is $O(\delta^{-2-(\gamma-\beta)/\alpha})$. Therefore, when $\beta_0 = \beta$, the computational costs of MC and MLMC are of the same order. Namely, MLMC does not accelerate and may even cost more than MC. In fact, we have identified a few such examples, such as hyperbolic equations with white noise parameters, as well as the random choice algorithm for the (deterministic) Jin-Xin model
\cite{jin1995relaxation} which is a hyperbolic system with relaxation. 

We list the results in \cref{tab:comparison} for the problems from literature \cite{abdulle2013stabilized,barth2011multi,cliffe2011multilevel,dereich2011multilevel,giles2008multilevel,lovbak2021multilevel} and the hyperbolic  problems that we study in this paper.

\begin{itemize}
    \item Case 1: stochastic differential equation (SDE) \cite{abdulle2013stabilized,dereich2011multilevel,giles2008multilevel};
    \item Case 2: random elliptic equation \cite{barth2011multi,cliffe2011multilevel};
    \item Case 3: particle method for transport equations \cite{lovbak2021multilevel}; 
    \item Case 4: scalar advection equation with random time dependent velocity $a(t, \omega)$, see \S~\ref{sec:comparison:scalar},
        \begin{itemize}
            \item 4.1, $a$ is piecewise constant in finite number of time intervals, 
            \item 4.2, $a$ is uniform white noise in time;
        \end{itemize} 
    \item Case 5: Euler equations with random time dependent adiabatic constant $\lambda(t, \omega)$, see \S~\ref{sec:comparison:euler},
            \begin{itemize}
            \item 5.1, $\lambda$ is piecewise constant in finite number of time intervals, 
            \item 5.2, $\lambda$ is uniform white noise in time;
        \end{itemize} 
    \item Case 6: shallow water equation with random bottom topography $B(x, \omega)$, see  \S~\ref{sec:comparison:shallowwater},
            \begin{itemize}
            \item 6.1, $B$ has a finite number of random parameters, 
            \item 6.2, $B$ has uniform white noise parameters;
        \end{itemize} 
    \item Case 7: random choice method for deterministic Jin-Xin model, see \S~\ref{sec:comparison:randomchoice},
            \begin{itemize}
            \item 7.1, semi-random case, only use random choice for the relaxation step,
            \item 7.2, fully-random case, use random choice for both the convection and the relaxation steps. 
        \end{itemize} 
\end{itemize}

\begin{table}[H]
\begin{center}
 \begin{tabular}{|l |c|c |c |c |c| c| c| c| c| c|} 
 \hline%
      \multicolumn{2}{|c|}{\multirow{2}{*}{Cases}}  & \multicolumn{2}{c|}{Dimension} & \multicolumn{4}{c|}{Parameters} & \multicolumn{2}{c|}{cost} & \multirow{2}{*}{Regime} 
    \\\cline{3-10}
      \multicolumn{2}{|c|}{} & physical & random ($K$) & $\alpha$ & $\beta_0$ & $\beta$ & $\gamma$ & MC & MLMC & 
    \\\hline
      \multicolumn{2}{|c|}{1}  & $t$ & $T/\Delta t$, in $t$ & 1 & 0 & 1 & 1 & $\delta^{-3}$ & $\delta^{-2}(\log\delta^{-1})^2$ & I
    \\\hline
      \multicolumn{2}{|c|}{2}  &  $x$ & $O(1)$, in $x$ & $1.75^*$  & $0^*$ & $2^*$ & 1 & $\delta^{-2.57}$ & $\delta^{-2}$ & I
    \\\hline
      \multicolumn{2}{|c|}{3}  &  $x,t$ & $T/\Delta t$, in $t$ & 1  & $0^*$ & 1 & 1 & $\delta^{-3}$ & $\delta^{-2}(\log\delta^{-1})^2$ & I
    \\\hline
      \multirow{2}{*}{4} & 4.1  & $x, t$ & $O(1)$, in $t$ & 1 & 0 & 2 & 2 & $\delta^{-4}$ & $\delta^{-2}(\log\delta^{-1})^2$ & I
    \\\cline{2-11}
       & 4.2  & $x,t$ & $T/\Delta t$, in $t$ & 1 & 1 & 1 & 2 & $\delta^{-3}$ & $\delta^{-3}$ & II
    \\\hline
      \multirow{2}{*}{5} & 5.1 & $x,t$  & $O(1)$, in $t$ & 1 & $0^*$ & $1.5^*$ & 2 & $\delta^{-4}$ & $\delta^{-2.5}$ & I
    \\\cline{2-11}
       & 5.2 & $x,t$  & $T/\Delta t$, in $t$ & 1 & $1^*$ & $1^*$ & 2 & $\delta^{-3}$ & $\delta^{-3}$ & II
    \\\hline
      \multirow{2}{*}{6} & 6.1 & $x,t$ & $O(1)$, in $x$ & 1 & $0^*$ & $1^*$ & 2 & $\delta^{-4}$ & $\delta^{-3}$ & I
    \\\cline{2-11}
       & 6.2  &  $x,t$ & $T/\Delta t$, in $x$ & 1 & $0^*$ & $0^*$ & 2 & $\delta^{-4}$ & $\delta^{-4}$ & III
    \\\hline
      \multirow{2}{*}{7} & 7.1 &  $x,t$ & $\left(T/\Delta t\right)^2$, in $x,t$ & 1 & $0.5^*$ & $0.5^*$ & 2 & $\delta^{-3.5}$  & $\delta^{-3.5}$ & II
    \\\cline{2-11}
       & 7.2 &  $x,t$ & $\left(T/\Delta t\right)^2$, in $x,t$ & 1 & $0^*$ & $0^*$ & 2 & $\delta^{-4}$ & $\delta^{-4}$ & III
    \\\hline
        \end{tabular}
\end{center}
    \caption{Comparison of MLML and MC. Parameters $\alpha$, $\beta_0$, $\beta$, and $\gamma$ are either analytically derived or obtained by numerical experiments which are marked with  a $*$. The complexities are computed using \cref{thm:costmlmc} and \cref{thm:costmc} in \S~\ref{sec:mlmc}, and numerically validated in literature or in the numerical experiments in \S~\ref{sec:comparison:scalar}, \S~\ref{sec:comparison:eulerandshallow} and \S~\ref{sec:comparison:randomchoice}. The expressions `in $x$', `in $t$', `in $x,t$', in the `random ($K$)' column indicate the dependence on the random parameters, $K$ is the number of random parameters (random dimension) in the problem, $T$ is the computational time, and $\Delta t$ is the time step of the numerical method. }
    \label{tab:comparison}
\end{table}


From above results in \cref{tab:comparison}, we can identify three regimes: 
\begin{itemize}
\item Regime I, $\beta > \beta_0=0$, for example, SDE and SPDEs with finite number of random parameters, where MLMC outperforms MC by orders of magnitude.
\item 
Regime II, $\beta_0 = \beta > 0$, where we can observe the decay of variances of MC solutions and MLMC corrections. The costs of both MLMC and MC are of the order $O(\delta^{-4+\beta})$, which is the case for the scalar advection equation and the Euler equations with uniform white noise parameters in time, and the semi-random choice method for the Jin-Xin model.
\item 
Regime III, $\beta_0 = \beta \simeq 0$, namely, the variances of MC solutions and MLMC corrections do not decay with respect to discretization parameters (time step and mesh size). The costs of both MLMC and MC are of the worst possible order $\delta^{-4}$, which is the case for the shallow water equation with uniform white noise parametrized topography and the fully-random choice method for the Jin-Xin model. 
\end{itemize}

\subsubsection*{Outline}
The rest of this paper is organized as follows. In \S~\ref{sec:mlmc}, we describe the principle and algorithm of MLMC, and present the main theorems -- the cost analysis of MLMC and standard MC methods. In \S~\ref{sec:comparison:scalar}, we first carry out an in-depth analysis for the scalar advection equation with a random velocity, then justify it with numerical experiments. In \S~\ref{sec:comparison:eulerandshallow} and \S~\ref{sec:comparison:randomchoice}, we provide comprehensive numerical experiments for other test examples, including hyperbolic equations with random parameters such as the Euler equations and the shallow water equations, and the random choice method for the multiscale Jin-Xin model. We conclude the paper in Section \S~\ref{sec:conclusion}. 

\subsubsection*{Notation}

%

The symbol $C$ (or $c$) denotes generic positive constants that may change from one line
of an estimate to the next. When estimating rates of decay or convergence, $C$
will always remain independent of approximation parameters. The dependence of $C$ will be clear from the context or stated explicitly. The expectation and variance of a random variable $X$ is abbreviated as $\E\left[X\right]$ and $\Var\left[X\right]$ respectively. The covariance of two random variables $X$, $Y$ is abbreviated as $\Cov\left[X, Y\right]$. To further simplify notation we will often write $\approx$ to mean that both sides of $\approx$ are equivalent infinitesimal.

\section{The Multilevel Monte Carlo method}
\label{sec:mlmc}

In this section, we first introduce the Multilevel Monte Carlo (MLMC) method and formulate the MC and MLMC algorithms in \S~\ref{sec:mlmc:formulation}.    
Then we present the theorems for their computational costs in \S~\ref{sec:mlmc:costanalysis}, which provide the crucial theoretical foundation of the paper. 

\subsection{Formulation and algorithm}
\label{sec:mlmc:formulation}

\subsubsection{MLMC formulation}
\label{sec:mlmc:formulation:formulation}

Given the time dependent hyperbolic problem with randomness, we denote $\Delta t_l$ as the time step on hierarchical levels $l=0,1,\dots,L$, with $\Delta t_{l-1}=\Gamma \Delta t_l$ with the integer refinement ratio $\Gamma>1$, and $\Delta x_l$ as the corresponding grid size which satisfies the CFL stability condition. 

For simplicity, we denote $P^{(l)}$ as the random numerical solution at final time $T$ computed using time step $\Delta t_l$. The QoI $\E\left[P^{(L)}\right]$ for MLMC can be expanded as a telescoping sum
\begin{equation}
    \E\left[P^{(L)}\right] = \E\left[P^{(0)}\right]+\sum_{l=1}^{L}\left(\E\left[P^{(l)}\right]-\E\left[P^{(l-1)}\right]\right).
    \label{eqn:MLAPMC-1}
\end{equation}

Instead of approximating $\E\left[P^{(L)}\right]$ on the finest level $L$, we can approximate the terms on the right hand side of \eqref{eqn:MLAPMC-1} by independent samples on hierarchical levels $l=0,1,\dots,L$, and sum them together. The name of the game is now how to reduce the sample variances.

To be more precise, let $P_l^{(l')}(s)$ be the $s$th sampled solution at final time $T$ computed using random variables on the $l$th level, as well as the time step $\Delta t_{l'}$ and the grid size $\Delta x_{l'}$ on the level $l'$, with $l'=l-1, l$. Let $N_l$ be the number of samples on the $l$th level. That is to say, we have samples $P_0^{(0)}(s)$ on the $0$th level, with $s=1, \dots, N_0$, and $P_l^{(l)}(s)$, $P_l^{(l-1)}(s)$ on the $l$th level, with $s=1, \dots, N_l$, $l=1,\dots,L$. 

We call $P_l^{(l)}- P_l^{(l-1)}$ as the MLMC correction. Since $P_l^{(l)}(s)$ and $P_l^{(l-1)}(s)$ are solutions on successive levels using the same random variables on the $l$th level, we have $\E\left[P_l^{(l)}- P_l^{(l-1)}\right] = \E\left[P^{(l)}\right]-\E\left[P^{(l-1)}\right]$. At the same time, we hope that $\Var\left[P_l^{(l)} - P_l^{(l-1)}\right]$ is small.


Hence, we have the following MLMC estimate
\begin{equation}\label{eqn:MLAPMC-2}
\begin{aligned}
    \hat{P}^{(L)} &:= N_0^{-1}\sum_{s=1}^{N_0}P_0^{(0)}(s) + \sum_{l=1}^{L}N_l^{-1}\sum_{s=1}^{N_l}\left(P_l^{(l)}(s)-P_l^{(l-1)}(s)\right) \\
    & \triangleq \hat{P}_0^{(0)} + \sum_{l=1}^{L} \left(\hat{P}_l^{(l)} - \hat{P}_l^{(l-1)}\right).
\end{aligned}
\end{equation}

\subsubsection{Optimal parameters for MLMC}
\label{sec:mlmc:formulation:formulation:optimalparameter}

In the MLMC estimate \eqref{eqn:MLAPMC-2}, one needs to choose the number of samples $N_l$ on hierarchical levels. The optimal value of $N_l$ can be determined in terms of the cost and variance of MLMC corrections by the following argument. 

Let $C_0$, $V_0$ be the computational cost and variance of $P_0^{(0)}$, and $C_l$, $V_l$ be the computational cost and variance of $P_ l^{(l)} - P_l^{(l-1)}$,  then the computational cost $\mathfrak{C}$ and the variance of $\hat{P}^{(L)}$ are
\begin{equation}\label{eqn:costvar}
        \mathfrak{C} = \sum_{l=0}^L N_l C_l, \quad \Var\left[\hat{P}^{(L)}\right] = \sum_{l=0}^L N_l^{-1} V_l.
\end{equation}

Let $R=\Var\left[\hat{P}^{(L)}\right]$ be the variance of $\hat{P}^{(L)}$, and our aim is to minimize the cost with respect to $N_l$. Using the Lagrange function $L$ with the multiplier $\mu^2$, we have
\begin{equation}\label{eqn:lagrange}
        L(N_l, \mu^2) = \sum_{l=0}^L N_l C_l + \mu^2\left(\sum_{l=0}^L N_l^{-1} V_l-R\right).
\end{equation}

The Kuhn-Tucker conditions (KKT) condition writes
\begin{equation}\label{eqn:KKT}
        \frac{\partial L(N_l, \mu^2)}{\partial N_l} = C_l - \mu^2 V_l N_l^{-2} = 0, \quad \Rightarrow N_l = \mu\sqrt{V_l/C_l},
\end{equation}
which means that the optimal $N_l$ is proportional to $\sqrt{V_l/C_l}$. By substituting \eqref{eqn:KKT} into $R=\Var\left[\hat{P}^{(L)}\right]$, one gets
\begin{equation}\label{eqn:muNvalue}
    \mu = R^{-1}\sum_{l=0}^{L} \sqrt{V_{l}C_{l}}, \quad N_l = R^{-1}\sqrt{V_l/C_l}\left(\sum_{l=0}^{L} \sqrt{V_{l}C_{l}}\right).
\end{equation}

\begin{remark}
  The variance $V_l$ is not known {\it a priori}, and one needs to estimate these values through samples in the implementation. 
\end{remark}

\subsubsection{An adaptive algorithm}
\label{sec:mlmc:formulation:formulation:adaptivealgorithm}

In this paper, we adopt the MLMC algorithm (as well as MC) with adaptively refined discretization levels as formulated in Giles' paper \cite{giles2008multilevel}. We decompose the $\mse$ into bias and variance, and require that both the squared bias and the variance of the estimator are bounded by $\delta^2/2$. This requirement leads to the adaptive criteria for the finest level $L$. The difference of MLMC and MC is that MLMC combines information from all the levels, while MC  only uses the samples on the finest level and the coarser levels are only used to determine the finest level. 

For simplicity, we denote $P$ as the exact solution at the final time $T$, $\hat{P}_l^{(l)}$, $\hat{P}_l^{(l-1)}$ as the estimated MLMC solutions on level $l$ with time step $\Delta t_l$, $\Delta t_{l-1}=\Gamma\Delta t_l$ respectively, and
\begin{equation}\label{def:Yl}
  \hat{Y}_l = \left\{
    \begin{aligned}
      &\hat{P}_0^{(0)}, & l = 0,\\
      &\hat{P}_l^{(l)}-\hat{P}_l^{(l-1)}, & l > 0,
    \end{aligned}\right.
\end{equation}
as estimated MLMC corrections. 

Given the accuracy threshold $\delta >0$ for the MLMC algorithm, we start from the initial $0$th level. Using an initial $N_0 = N_I$ number of samples, we obtain the Monte Carlo estimates $\hat{V}_0$, $\hat{Y}_0$ of $V_0$, $Y_0$ respectively. Then the updated number of samples on $0$th level $N_{0}^{\prime}$ can be determined by $N_{0}^{\prime} = \left\lceil 2\delta^{-2} \hat{V}_{0} \right\rceil$. If $N_0^\prime > N_0$, we evaluate extra samples and update $\hat{V}_0$, $\hat{Y}_0$, $N_0 = N_0^\prime$.

For the iteration $l\rightarrow l+1$, $V_{l+1}$, $Y_{l+1}$ can be estimated by $\hat{V}_{l+1}$, $\hat{Y}_{l+1}$ at the new level $l+1$ using $N_{l+1}=N_I$ samples, then the optimal number of samples on the $k$th level $N_{k}^\prime$, with $k = 0, \dots, l+1$, can be determined by
\begin{equation}\label{eqn:detN}
    N_{k}^\prime = \left\lceil 2\delta^{-2} \sqrt{\hat{V}_{k}/C_{k}}\left( \sum_{j=0}^{l+1} \sqrt{\hat{V}_j C_j} \right) \right\rceil,
\end{equation}
where $C_j$ is the computational cost on level $j$. 
If $N_{k}^\prime > N_{k}$, we let $N_{k}=N_{k}^\prime$, evaluate extra samples and update $\hat{V}_{k}$, $\hat{Y}_{k}$. 

We can identify the finest level $L$ by an a posterior error estimate through the requirement that the squared bias of the estimator is bounded by $\delta^2/2$. For a numerical scheme with (weak) convergence rate $\alpha$, we have
\begin{equation}
    \E\left[\hat{P}_{l}^{(l)} - P\right] = \E\left[ P_{l}^{(l)} - P \right] \approx c_1 (\Delta t_l)^\alpha.
\end{equation}

The following estimate of the correction $\E\left[\hat{P}_l^{(l)}-\hat{P}_l^{(l-1)}\right]$ 
\begin{equation}
    \E\left[\hat{P}_l^{(l)}-\hat{P}_l^{(l-1)}\right] \approx ( 1 - \Gamma^\alpha) c_1 (\Delta t_l)^\alpha \approx ( 1 - \Gamma^\alpha) \E\left[ \hat{P}_l^{(l)} - P\right],
\end{equation}
can be used to obtain the stopping criteria
\begin{equation}\label{eqn:criteria}
    \left|\hat{Y}_l\right| = \left|\hat{P}_l^{(l)}-\hat{P}_l^{(l-1)}\right| \leq \frac{\Gamma^\alpha-1}{\sqrt{2}}\delta,
\end{equation}
which leads to $\left|\E\left[\hat{P}_l^{(l)} - P\right]\right|^2 \leq \delta^2/2$. 

We summarize the adaptive MLMC algorithm \cite{giles2008multilevel} as follows.

\begin{algorithm}[H]
  \caption{Multilevel Monte Carlo method}\label{alg:mlmc}
  Given $N_{I}\in \mathbb{N} $. 
  \begin{algorithmic}[1]
  \FOR{$l = 0, 1, \dots$}
  \STATE Draw $N_{l} = N_I$ samples to obtain $\hat{V}_l$, $\hat{Y}_l$;
  \STATE Define optimal $N_k^\prime$, $k=0,\dots,l$ using \eqref{eqn:detN};
  \STATE If $N_{k}^\prime>N_k$, generate extra samples and update $\hat{V}_k$, $\hat{Y}_k$, $N_k=N_k^\prime$ for $k=0,\dots,l$;
  \IF{$l\geq 1$ \& \eqref{eqn:criteria} is \TRUE}
    \item Stop with $L=l$;
  \ENDIF
  \STATE $l = l+1$;
  \ENDFOR
  \end{algorithmic}
\end{algorithm}

For MC, we only consider the solution $\hat{P}^{(L)} = \hat{P}_{L}^{(L)}$ on the finest level. Denoting $\mathsf{V}_l = \Var\left[P_{l}^{(l)}\right]$, we require  
\begin{equation}
\Var\left[\hat{P}_{L}^{(L)}\right] = N_{L}^{-1}\mathsf{V}_{L} \leq \delta^2/2,
\label{eqn:varmc}
\end{equation}
in MC. For comparison, we require $\Var\left[\hat{P}^{(L)}\right] = \sum_{l=0}^{L} N_{l}^{-1}V_{l} \leq \delta^2/2$ in MLMC. 

By \eqref{eqn:varmc}, the optimal number of samples $N_l^\prime$ can be taken as
\begin{equation}\label{eqn:detNmc}
    N_l^\prime = \left\lceil 2\delta^{-2}\hat{\mathsf{V}}_{l}  \right\rceil,
\end{equation}
where $\hat{\mathsf{V}}_{l}$ is the Monte Carlo estimate of $\mathsf{V}_{l}$ using initial $N_I$ samples.
We use a similar a posterior estimate as the stopping criteria
\begin{equation}\label{eqn:criteriamc}
    \left|\hat{P}_l^{(l)}-\hat{P}_{l-1}^{(l-1)}\right| < \frac{\Gamma^\alpha-1}{\sqrt{2}}\delta.
\end{equation}

The adaptive MC algorithm can be given in the spirit of adaptive MLMC \cref{alg:mlmc}.

\begin{algorithm}[H]
  \caption{Monte Carlo method}\label{alg:mc}
  Given $N_I\in \mathbb{N} $. 
  \begin{algorithmic}[1]
  \FOR{$l = 0, 1, \dots$}
  \STATE Draw $N_l = N_{I}$ samples to obtain $\hat{\mathsf{V}}_l$, $\hat{P}_l^{(l)}$;
  \STATE Define $N_l^\prime$ using \eqref{eqn:detNmc};
  \STATE If $N_l^\prime > N_l$, generate extra samples at level $l$ and update $\hat{\mathsf{V}}_l$, $\hat{P}_l^{(l)}$, $N_l=N_l^\prime$;
  \IF{$l\geq 1$ \& \eqref{eqn:criteriamc} is \TRUE}
    \item Stop with $L = l$;
  \ENDIF
  \STATE $l = l+1$;
  \ENDFOR
  \end{algorithmic}
\end{algorithm}

\subsection{Cost analysis}
\label{sec:mlmc:costanalysis}

In this section, we present cost theorems for MLMC and MC methods, as prescribed in \cref{alg:mlmc} and \cref{alg:mc}, which are crucial for the hyperbolic equation applications. The proofs are given in Appendix \ref{sec:appendix:costanalysis:costMLMC} and \ref{sec:appendix:costanalysis:compleixtyMC}, respectively.

\begin{theorem}[MLMC Computational Cost]
Assume that the MLMC estimators $\hat{Y}_l$ ($l\in\mathbb{N}$) defined in \eqref{def:Yl} has the following properties:
\begin{enumerate}[{\textbf{A}}1.]
  \item\label{Assum:1-1} $\E\left[P^{(l)}-P\right] \leq c_1 (\Delta t_l)^{\alpha}$,
  \item\label{Assum:1-2} $\Var\left[\hat{Y}_l\right] = N_l^{-1}V_l \leq c_2 N_l^{-1} (\Delta t_l)^{\beta}$,
  \item\label{Assum:1-3} $\mathfrak{C}\left[\hat{Y}_l\right]=N_lC_l \leq c_3 N_l (\Delta t_l)^{-\gamma}$,
\end{enumerate}
where $2\alpha\geq \gamma$, $\beta$, $\gamma$, $c_1$, $c_2$, $c_3$ are positive constants. There exist $L\in\mathbb{N}^+$, $N_l\in \mathbb{N}^+$, $l=0, ..., L$, and a constant $c_4>0$ such that for any $\delta < e^{-1}$, the multilevel estimator
\begin{equation}
  \hat{P}^{(L)} = \sum_{l=0}^L \hat{Y}_l,
\end{equation}
has the following $\mse$ bound
\begin{equation}
  \mse \equiv \E\left[\left(\hat{P}^{(L)}-\E\left[P\right]\right)^2\right] < \delta^2,
\end{equation}
and the computational cost $\mathfrak{C}$ is bounded by
\begin{equation}
  \begin{aligned}
  \mathfrak{C} &\leq \left\{ \begin{aligned}  
  & c_4\delta^{-2}, & \beta > \gamma, \\
  & c_4\delta^{-2}(\log \delta^{-1})^2, & \beta = \gamma,\\
  & c_4\delta^{-2-(\gamma-\beta)/\alpha}, & 0 < \beta < \gamma.
  \end{aligned}\right.
  \end{aligned}
\end{equation}
\label{thm:costmlmc}
\end{theorem}

\begin{theorem}[MC Computational Cost]
Assume that the MC estimators $\hat{P}_l^{(l)}$ ($l\in\mathbb{N}$) has the following properties:
\begin{enumerate}[{\textbf{A}}1.]\setcounter{enumi}{3}
  \item\label{Assum:2-1} $\E\left[P^{(l)}-P\right] \leq c_1 (\Delta t_l)^{\alpha}$,
  \item\label{Assum:2-2} $\Var\left[\hat{P}_l^{(l)}\right] = N_l^{-1}\mathsf{V}_l \leq c_2 N_l^{-1} (\Delta t_l)^{\beta_0}$,
  \item\label{Assum:2-3} $\mathfrak{C}\left[\hat{P}_l^{(l)}\right]=N_l\mathfrak{C}\left[P_{l}^{(l)}\right] \leq c_3 N_l (\Delta t_l)^{-\gamma}$,
\end{enumerate}
where $2\alpha\geq \gamma$, $\beta_0$, $\gamma$, $c_1$, $c_2$, $c_3$ are positive constants. There exist $L\in\mathbb{N}^+$, $N_l\in \mathbb{N}^+$, $l=0, ..., L$, and a constant $c_4>0$ such that for any $\delta < e^{-1}$, the MC estimator $\hat{P}_L^{(L)}$ has the following $\mse$ bound
\begin{equation}
  \mse \equiv \E\left[\left(\hat{P}_L^{(L)}-\E\left[P\right]\right)^2\right] < \delta^2,
  \label{eqn:mse}
\end{equation}
and the computational cost $\mathfrak{C}^*$ is bounded by
\begin{equation}
  \begin{aligned}
  \mathfrak{C}^* &\leq \left\{ \begin{aligned}  
  & c_4\delta^{-2}, & \beta_0 > \gamma, \\
  & c_4\delta^{-2}\log \delta^{-1}, & \beta_0 = \gamma,\\
  & c_4\delta^{-2-(\gamma-\beta_0)/\alpha}, & 0 < \beta_0 < \gamma.
  \end{aligned}\right.
  \end{aligned}
\end{equation}
\label{thm:costmc}
\end{theorem}

Comparing the results of \cref{thm:costmlmc} and \cref{thm:costmc}, we can see that when $0<\beta_0<\beta<\gamma$, the MLMC method outperforms the MC method.

\section{ Comparing MLMC and MC for the scalar advection equation with random velocity}
\label{sec:comparison:scalar}

In sections \S~\ref{sec:comparison:scalar}, \S~\ref{sec:comparison:eulerandshallow} and \S~\ref{sec:comparison:randomchoice}, we use four examples to compare MLMC and MC methods for both linear and nonlinear problems. In this section, we consider the first example, the scalar advection equation with a time dependent random velocity. For this example, we analyze the variances of the MC solutions and MLMC corrections, and provide the complexity estimates of MLMC and MC methods in \cref{thm:scalar-cost} for two cases: 1) the number of random parameters is finite, 2) the random parameter is modeled as white noise. In the first case, MLMC does outperform MC, while in the second case, MLMC has the same order of computational cost $\delta^{-3}$ as MC. We then use numerical experiments to justify those theoretical results.


We first consider the scalar advection equation
\begin{equation}\label{eqn:scacon}
    u_t + a(t,\omega)u_x = 0,
\end{equation}
with $x\in[-1,1]$, $t \in [0,T]$, initial condition $u(0,x) = u_0(x)$, and periodic boundary condition $u(-1,t) = u(1,t)$. 

We assume that the random variable is of the form $\omega = (\omega_1, \dots, \omega_K)$, such that $K$ is the random dimension. The random time dependent velocity $a=a(t, \omega)$ can be modeled as piecewise constants over $K$ equal length time intervals
\begin{equation}
    a(t, \omega) = \bar{a} + \omega_k \quad  \text{ for } \quad \frac{k-1}{K}T \leq t < \frac{k}{K}T, \quad k=1,\dots,K,
\end{equation}
where $\bar{a}=1$, $\omega_k\sim\mathcal{U}(-1,1)$. The corresponding solution is denoted as $u(x, t, \omega)$.

Let $a(t, \omega(s))$ be a sample of $a(t, \omega)$. We denote $v^{n}_{i}(s)$ as the numerical solution at $(x_i, t_n)$, with the random variable $\omega(s)$, $x_i=-1+i\Delta x_l$, $t_n=n\Delta t_l$, where $\Delta t_l$ is the time step and $\Delta x_l$ is the mesh size, with $\frac{\Delta t_l}{\Delta x_l} = \kappa \leq 1$ by the CFL condition. In the following, when no confusion arises, we also use notations $v^{n}_{i}$ or $v^n$ for simplicity. We denote $X_l:=\frac{2}{\Delta x_l}$ and $M_l:=\frac{T}{\Delta t_l}$ as the spatial and temporal degrees of freedom.

Equation \eqref{eqn:scacon} can be discretized by the first order upwind scheme
\begin{equation}
    \frac{v_i^{n+1}(s) - v_i^n(s)}{\Delta t} + a(t_n,\omega{(s)}) \frac{v_i^{n}(s) - v_{i-1}^n(s)}{\Delta x} = 0, \quad a>0.
    \label{eqn:upwind}
\end{equation}

\subsection{Analysis}
\label{sec:comparison:scalar:analysis}

We now carry out an analysis for this simple model. 
Denoting $u(\cdot, t,\omega(s))$ as the exact solution of \eqref{eqn:scacon} using the random variable $\omega(s)$, which can be obtained by method of characteristics. At the final time $T$, we have the following expression
\begin{equation}
    u\left(x, T,\omega{(s)}\right) = u\left(x-\left(\bar{a}+\omega_K{(s)}\right) \frac{T}{K}, \frac{K-1}{K}T \right) = \dots = u_0\left(x-\bar{a}T-\frac{\sum_{k=1}^K\omega_k{(s)}}{K}T\right).
\end{equation}

\begin{proposition}
\label{prop:propertyofvar}
Denote $P = u\left(\cdot, T,\omega(s)\right)$ as the exact solution and $P^{(l)}$ as the numerical solution using time step $\Delta t_l$ and mesh size $\Delta x_l$ at final time $T$. The variances of $P$, $P^{(l)}$ satisfy the following properties
\begin{enumerate}[{\textbf{P}}1.]
    \item\label{property:varP} $\Var\left[P\right] := \sum_{i} \Var\left[P(x_i)\right] \Delta x_l  = O(K^{-1})$,
    \item\label{property:varPlP} $\Var\left[P^{(l)} - P\right] := \sum_{i} \Var\left[(P^{(l)} - P)(x_i)\right] \Delta x_l  = O((\Delta t_l)^{2})$,
    \item\label{property:varPl} $\Var\left[P^{(L)}\right] := \sum_{i} \Var\left[P^{(l)}(x_i)\right] \Delta x_l  = O(K^{-1})$.
\end{enumerate}
\end{proposition}

\begin{proof}
\begin{enumerate}[(a)]
    \item Proof of \textit{\textbf{P}}\ref{property:varP}. Defining $\bar{u}=u_0(x-\bar{a}T)$, we calculate $P - \bar{u}$ as follows
\begin{equation}
    \left|(P - \bar{u})(x_i)\right|^2 = \left|u(x_{i}, T, \omega(s))-\bar{u}(x_i)\right|^2 \leq 
    \left\|u_0^\prime(\cdot)\right\|_{\infty}^2 T^2 \left|\frac{\sum_{k=1}^K\omega_k^{(s)}}{K}\right|^2,
\end{equation}
and 
\begin{equation}
    \E\left[\left|\frac{\sum_{k=1}^K\omega_k^{(s)}}{K}\right|^2\right] = \Var\left[\frac{\sum_{k=1}^K\omega_k^{(s)}}{K}\right] = \frac{\Var\left[\omega_1\right]}{K}.
\end{equation}

Property \textit{\textbf{P}}\ref{property:varP} then follows by
\begin{equation}
\Var\left[P\right] = \sum_{i}\Var\left[\left(P-\bar{u}\right)(x_i)\right]\Delta x_l \leq \sum_{i} \E\left[\left(P-\bar{u}\right)^2(x_i) \right] \Delta x_l = O(K^{-1}).
\end{equation}

\item Proof of \textit{\textbf{P}}\ref{property:varPlP}. We define the following intermediate variable with respect to the upwind scheme \eqref{eqn:upwind},
$$\tilde{u}(x_i, t_{n+1}, \omega(s)) : = u(x_i, t_{n}, \omega{(s)}) - a(t_n, \omega(s)) \frac{\Delta t_l}{\Delta x_l}(u(x_i, t_{n}, \omega{(s)}) - u(x_{i-1}, t_{n}, \omega{(s)})).
$$
The local truncation error of the upwind scheme is
\begin{equation}\label{eqn:scalaranaly1}
\begin{aligned}
    \tau_{i}^{n} & := u(x_i, t_{n+1}, \omega{(s)}) - \tilde{u}(x_i, t_{n+1}, \omega{(s)}) \\
    &= \frac{a(t_n, \omega(s))(\Delta t_l)^2}{2}\left(a(t_n, \omega(s)) u_{xx}(x_i, t_{\eta}) - \frac{\Delta x_l}{\Delta t_l} u_{xx}(x_{\xi}, t_n)\right) \\
    &\leq \frac{\mathcal{C}_1 a(t_n, \omega(s))}{2}\left(a(t_n, \omega(s)) + \frac{\Delta x_l}{\Delta t_l}\right) (\Delta t_l)^2 \\
    &\leq \left(2+\frac{1}{\kappa}\right) \mathcal{C}_1 (\Delta t_l)^2,
\end{aligned}
\end{equation}
where $\mathcal{C}_1 = \sup|u_{xx}|$. Then 
$$\left\|\tau^n\right\|_2 \leq \left(2+\frac{1}{\kappa}\right) \mathcal{C}_1 \sqrt{X_{l}} (\Delta t_l)^2.$$ 

The stability of the upwind scheme is obvious since \begin{equation}\label{eqn:scalarstability}
    \left\| v^{n+1} \right\|_2^2 = \sum_{i} \left|v_{i}^{n+1}\right|^2 \leq \sum_{i} (1-\kappa a(t_n, \omega(s))) \left|v_{i}^{n}\right|^2 + \kappa a(t_n, \omega(s)) \left|v_{i-1}^{n}\right|^2 = \left\| v^{n} \right\|_2^2,
\end{equation}
where the CFL condition $\kappa\le 1$ is used. The error $P^{(l)}-P$ can be bounded by
\begin{equation}\label{eqn:scalaranaly2}
\begin{aligned}
    \left\|P^{(l)}-P\right\|_2 &= \left\| v^{M_l} - u\left(\cdot, t_{M_l}, \omega(s)\right) \right\|_2 \\
    &\leq \left\| v^{M_l} - \tilde{u}\left(\cdot, t_{M_l}, \omega(s)\right) \right\|_2 + \left\| u\left(\cdot, t_{M_l}, \omega(s)\right) - \tilde{u}\left(\cdot, t_{M_l}, \omega{(s)}\right) \right\|_2 \\
    (\text{stability}) &\leq  \left\| v^{M_l-1} - u\left(\cdot, t_{M_l-1}, \omega(s)\right) \right\|_2 + \left\|\tau^{M_l-1}\right\|_2 \\
    &\leq \cdots \\
    &\leq 0 + \sum_{j=0}^{M_l-1} \left\|\tau^j\right\|_2 \\
    &\leq \left(2+\frac{1}{\kappa}\right) \mathcal{C}_1 T \sqrt{X_l}\Delta t_l.
\end{aligned}
\end{equation}
Property  \textit{\textbf{P}}\ref{property:varPlP} follows by
\begin{equation}\label{eqn:scalaranaly3}
    \Var\left[P^{(l)}-P\right] \leq \sum_{i} \E\left[\left(P^{(l)}-P\right)^2(x_i)\right] \Delta x_l = \E\left[\left\|P^{(l)}-P\right\|_2^2\right] \Delta x_l = O((\Delta t_l)^2).
\end{equation}

\item Proof of \textit{\textbf{P}}\ref{property:varPl}. Property  \textit{\textbf{P}}\ref{property:varPl} follows by the assumption that $\Delta t_l = O(K^{-1})$, and the inequality 
$$
\left( \sqrt{\Var\left[P^{(l)}-P\right]} - \sqrt{\Var\left[P\right]} \right)^2 \leq \Var\left[P^{(l)}\right] \leq \left(\sqrt{\Var\left[P^{(l)}-P\right]} + \sqrt{\Var\left[P\right]}  \right)^2.
$$
\end{enumerate}

\end{proof}

We compute MLMC corrections, which are the difference of solutions $P^{(l)}_l$ and $P^{(l-1)}_l$, the numerical solutions at final time $T$ of \eqref{eqn:scacon} on the $l$th level, with time step $\Delta t_l$, $\Delta t_{l-1}=\Gamma\Delta t_{l}$ respectively. We have the following theorem for the computational costs of MLMC and MC for the scalar random advection equation. 

\begin{theorem}
\label{thm:scalar-cost}
We consider two cases,
  \begin{itemize}
      \item
      Case I: When the number of random variables $K$ is fixed, and $M_{0}\geq K$, we have $\beta_0 = 0$, $\beta = 2$, i.e.
      \begin{equation}
          \Var\left[P_l^{(l)}\right] = O(1), \quad  \Var\left[P_l^{(l)}-P_{l}^{(l-1)}\right] = O((\Delta t_l)^2).
          \label{eqn:example1case1}
      \end{equation}
      Given the accuracy threshold $\delta>0$, the cost of MLMC is $O\left(\delta^{-2}\left(\log\delta^{-1}\right)^2\right)$ and the cost of MC is $O\left(\delta^{-4}\right)$.
      
      \item
      Case II: The random velocity is modeled as white noise, namely the number of random variables depends on $l$, $K_l = \frac{T}{\Delta t_l}= M_l$ for each level $l$. We have $\beta_0 = \beta = 1$, i.e.
      \begin{equation}
          \Var\left[P_l^{(l)}\right] = O(\Delta t_l), \quad  \Var\left[P_l^{(l)}-P_{l}^{(l-1)}\right] = O(\Delta t_l).
      \end{equation}
      Given the accuracy threshold $\delta>0$, the computational costs of MLMC and MC are both $O\left(\delta^{-3}\right)$.
  \end{itemize}
\end{theorem}

\begin{proof}
  The values of $\beta_0$ can be obtained directly from Property \textit{\textbf{P}}\ref{property:varPl}. In the following, we will explain how random variables are used in both cases and evaluate $\beta$.
  \begin{itemize}
    \item Proof of Case I.
    We first generate $K$ random variables for one sample on level $l$. Since $K$ is fixed, we use the same random variables $\left\{\omega_k(s)\right\}_{k=1}^{K}$ to calculate $P_l^{(l)}(s)$ and $P_{l}^{(l-1)}(s)$. Therefore similar error analysis as in \eqref{eqn:scalaranaly1}, \eqref{eqn:scalaranaly2}, \eqref{eqn:scalaranaly3} leads to the following estimate
    \begin{equation}\label{eqn:scavar2}
        \Var\left[P_l^{(l)}-P_{l}^{(l-1)}\right] \leq \E\left[ \left\| P_l^{(l)}-P_{l}^{(l-1)} \right\|_2^2\right] = O((\Delta t_l)^2).
    \end{equation}
    Thus Property \textit{\textbf{A}}\ref{Assum:1-2} holds with $\beta = 2$. The computational costs of MLMC and MC follow by the application of \cref{thm:costmlmc}, \cref{thm:costmc}, and the fact that $\alpha=1$, $\gamma=2$, $\beta_0=0$, $\beta=2$.
    
    In this case, $\left\{\omega_k(s)\right\}_{k=1}^{K}$ can be directly used to calculate $P_{l}^{(l-1)}$, and the numerical implementation of \cref{alg:mlmc} is straightforward.
    
    \item Proof of Case II.
    We need to generate $K_l = T / \Delta t_l$ random variables for one sample on level $l$, denoted by $\left\{\omega_k\right\}_{i=1}^{K_l}$. 
    For example, let $a$ be a uniform white noise
    \begin{equation}
        a(t) = \bar{a} + \omega(t), \quad \omega(t) \sim \mathcal{U}(-1,1),
    \end{equation}
    with correlation $B(t) = \Var\left[\omega\right] \delta(t) = \frac{1}{3} \delta(t)$. At each time $t_{n-1}$, $a(t_{n-1}) = \bar{a} + \omega_{n}$, $n = 1,\dots,K_l$. The computation of $P^{(l)}_l$ is also straightforward, while the computation of $P^{(l-1)}_l$ is a little tricky. To compute $P^{(l-1)}_l$, we need to construct $K_{l-1}$ random variables from the $K_{l}$ random variables, such that
    \begin{enumerate}
        \item 
        They are strongly correlated with $\left\{\omega_k\right\}_{k=1}^{K_l}$, so that $\Var\left[P^{(l)}_l - P^{(l-1)}_l\right]$ is kept small; 
        \item 
        The distribution of the coarse random variables is the same as its finer counterpart, so that $\E\left[P^{(l-1)}_l\right] = \E\left[P^{(l-1)}_{l-1}\right]$ still holds.
    \end{enumerate}
    
To that end, we use the following trick \cite{lovbak2021multilevel} to construct those coarse random variables
    \begin{equation}\label{eqn:combinetrick}
        \frac{\tilde{a}(t_{(n-1)\Gamma})}{2} = \left(\max_{1 \leq j \leq \Gamma}\left\{\frac{a(t_{ n\Gamma- j})}{2}\right\}\right)^\Gamma \sim \mathcal{U}(0,1), \quad n = 1,\dots, K_{l-1} (= K_l/\Gamma).
    \end{equation}
    
The inequality \eqref{eqn:scavar2} does not work since the random velocities $\left\{\tilde{a}(t_{n-1})\right\}_{n=1}^{K_{l-1}}$, $\left\{a(t_{n-1})\right\}_{n=1}^{K_l}$ are not the same now, although strongly correlated. Nevertheless, the Cauchy-Schwarz inequality implies that
    \begin{equation}\label{eqn:scavar4}
        \Var\left[P_l^{(l)}-P_{l}^{(l-1)}\right] \leq \left( \sqrt{\Var\left[P_l^{(l)}\right]} + \sqrt{\Var\left[P_l^{(l-1)}\right]} \right)^2 = O(\Delta t_l).
    \end{equation}
    
The Property \textit{\textbf{A}}\ref{Assum:1-2} holds with $\beta = 1$. The computational costs of MLMC and MC follow by the application of \cref{thm:costmlmc}, \cref{thm:costmc}, and the fact that $\alpha=1$, $\gamma=2$, $\beta_0=1$, $\beta=1$.

  \end{itemize}
\end{proof}

\subsection{Numerical results: Case I}
\label{sec:comparison:scalar:casei}

In the numerical experiments for Case I, we choose two constants $K=1$ and $K=32$. We take
\begin{equation}\label{eqn:params}
    \kappa=\frac{\Delta t_l}{\Delta x_l} = 0.5,\quad \Gamma =  2,\quad \Delta x_0 = \frac{1}{32}, \quad N_{I} = 500,
\end{equation}
where $\Gamma$ is the refinement ratio, $N_{I}$ is the initial number of sampling, the CFL condition is satisfied since $|a|\leq 2$. We also choose the initial condition
\begin{equation}
    u_0(x) =\frac{\sin(x)+1}{2},
\end{equation}
and the periodic boundary condition.

In view of our theoretical results, we will show the computational costs of MC and MLMC vs. the accuracy level $\delta$, also their variances vs. level. 
Let $N_l$, $N_l^\ast$ represent the number of samples of MLMC and MC respectively. The computational cost of MLMC can be estimated as the total number of time steps and gird points on all levels
normalized by that on the level $0$,
\begin{equation}
    \mathfrak{C}=N_{0}+\sum_{l=1}^{L} N_{l}\left(\Gamma^{2 l}+\Gamma^{2 (l-1)}\right).
\end{equation}
Namely, the computational cost of one sample on level $0$ ($C_0$) is regarded as $1$ unit, and $C_{l+1}/C_{l} = \Gamma^2, l\geq 1$ since we refine the grid in both spatial and temporal dimensions. 

The computational cost of MC can be estimated similiarly as
\begin{equation}
    \mathfrak{C}^\ast=N_{0}^\ast+\sum_{l=1}^{L} N_{l}^\ast \Gamma^{2 l}.
\end{equation}

In \cref{fig:scamlmc1} and \cref{fig:scamlmc2}, the numerical results are shown in solid lines, while the theoretical order is plotted in dashed lines. We observe that as long as $K$ is a constant, we have $\Var\left[P_l^{(l)}\right]=O(1)$ and $\Var\left[P_l^{(l)}-P_{l}^{(l-1)}\right] = O((\Delta t_l)^2)$, i.e. $\beta_0=0$, $\beta=2$. The costs of MC and MLMC are of order $O(\delta^{-4})$ and $O(\delta^{-2})$, respectively. In this case, to achieve the same accuracy, MLMC outperforms MC by two orders of magnitude in terms of the computational cost.

\begin{enumerate}
    \item $K=1$
    \begin{figure}[H]
      \subfigure{
        \includegraphics[width=0.45\textwidth]{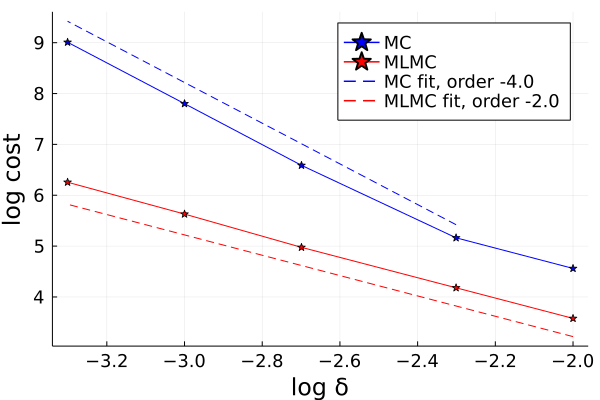}}
      \subfigure{
        \includegraphics[width=0.45\textwidth]{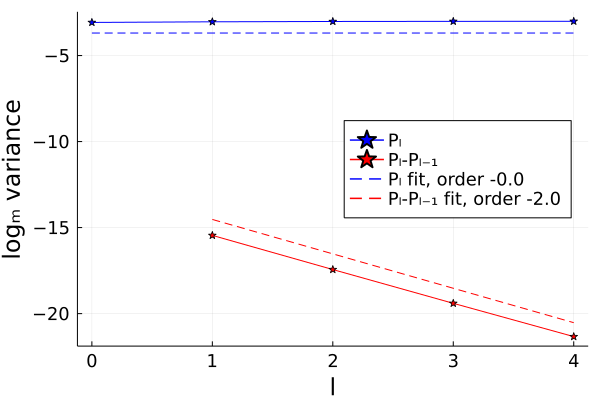}}
      \caption{Example 1: Case I, $K=1$, the computational cost w.r.t accuracy threshold $\delta$ (left) and variance on different levels (right).}\label{fig:scamlmc1}
    \end{figure}

    \item $K=32$
    \begin{figure}[H]
      \subfigure{
        \includegraphics[width=0.45\textwidth]{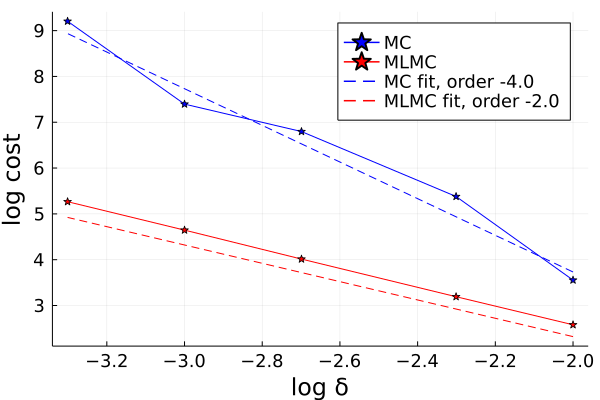}}
      \subfigure{
        \includegraphics[width=0.45\textwidth]{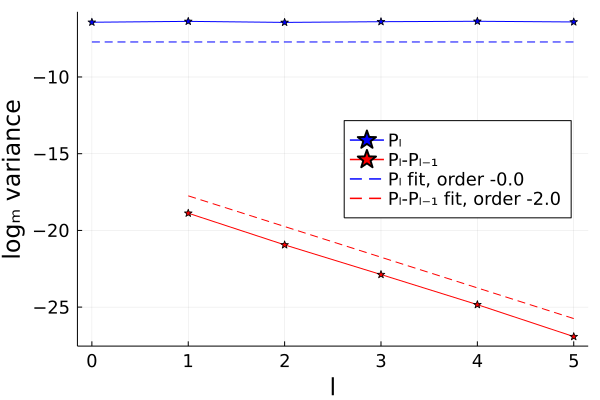}}
      \caption{Example 1: Case I, $K=32$, the computational cost w.r.t accuracy threshold $\delta$ (left) and variance on different levels (right).}\label{fig:scamlmc2}
    \end{figure}

\end{enumerate}

\subsection{Numerical results: Case II}
\label{sec:comparison:scalar:caseii}

In this case, we need to generate $K_l=T/\Delta t_l$ random variables $\omega_k\sim\mathcal{U}(-1,1)$ on the $l$ th level. Using \eqref{eqn:combinetrick}, we can construct $K_{l-1}$ random variables and use them to calculate $P_l^{(l-1)}$. \cref{fig:scamlmc3} shows that both $\Var\left[P_l^{(l)}\right]$ and $\Var\left[P_l^{(l)}-P_{l}^{(l-1)}\right]$ decays with order $O(\Delta t_l)$, i.e. $\beta_0=\beta=1$, which is consistent with the derivation in \eqref{eqn:scavar4}. The costs of MC and MLMC are both of the order $O(\delta^{-3})$, while MLMC appears to have a larger constant than MC. 
    
    \begin{figure}[H]
      \subfigure{
        \includegraphics[width=0.45\textwidth]{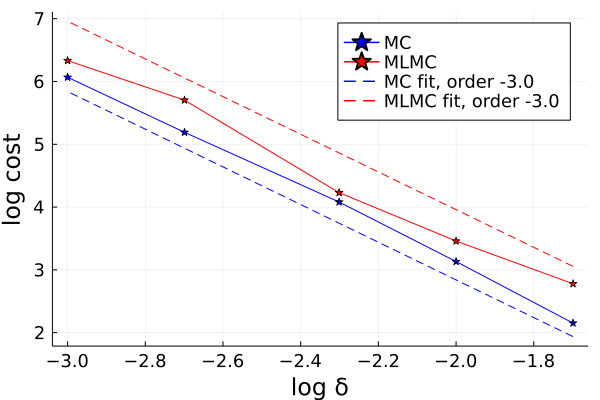}}
      \subfigure{
        \includegraphics[width=0.45\textwidth]{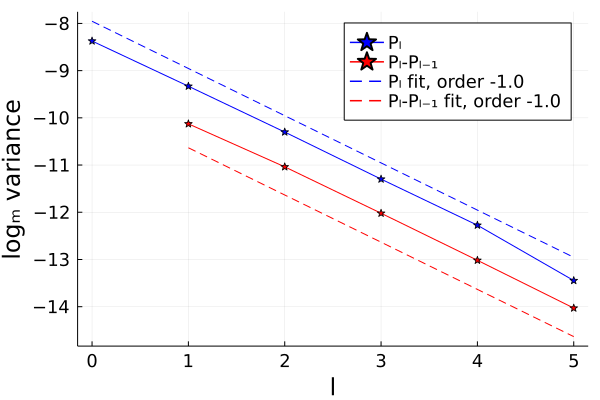}}
  \caption{Example 1: Case II, the computational cost w.r.t accuracy threshold $\delta$ (left) and variance on different levels (right).}\label{fig:scamlmc3}
\end{figure}

\section{Comparing MLMC and MC for the Euler equation and the shallow water equation}
\label{sec:comparison:eulerandshallow}

In this section, we extend our study to more practical examples such as the Euler equation with a time dependent random adiabatic constant, and the shallow water equations with space dependent random tomography. The Euler equation example is qualitatively similar to the scalar advection example, albeit more complex. However, the shallow water equation example is qualitatively distinct, and has computational cost $\delta^{-4}$ if the random tomography is modeled as white noise. In \cite{mishra2012mlmcconservationlaw}, Mishra et. al. applied MLMC to Euler equations with random initial data, while we consider Euler equations with time dependent random parameters in \S~\ref{sec:comparison:euler}.  Mishra et. al. also applied MLMC to shallow water equations with uncertain topography in \cite{mishra2012mlmcshallow}, where the bottom topography is treated as a fixed number of random variables with uniform distribution. That is similar to our Case I, but we also consider the white noise Case II in  \S~\ref{sec:comparison:shallowwater}.


\subsection{Euler equation with random adiabatic constant $\lambda(t, \omega)$}
\label{sec:comparison:euler}

In this subsection, we consider the Euler equations as a prototypical nonlinear partial differential equations with time dependent random parameters. When the number of random variables $K$ is fixed, the analysis of Property \textit{\textbf{A}}\ref{Assum:1-2} is similar as \eqref{eqn:scavar2}; while the white noise case is now hard to analyze due to nonlinearity, we can numerically estimate $\Var\left[P_l^{(l)}\right]$, $\Var\left[P_l^{(l)}-P_{l}^{(l-1)}\right]$ and their decay orders $\beta_0$, $\beta$.

Given the  Euler equations with gravity
\begin{equation}\label{eqn:Euler}
    \left\{\begin{aligned} \rho_t + (\rho v)_x &= 0, \\
    (\rho v)_t + (v\rho v)_x + p_x &= -\rho\phi_x, \\ 
    E_t + [(E+p)v]_x &= -\rho v \phi_x, \end{aligned}\right.
\end{equation}
where
$$
E = \rho e + \frac{\rho}{2}v^2, \quad \phi(x) = gx.
$$
Here $\rho(t,x)$ is the density, $v(t,x)$ is the velocity, $E(t,x)$ is the total energy, $p(t,x)$ is the pressure, $e(t,x)$ is the specific internal energy, and $g$ is the gravitational constant. The adiabatic constant $\lambda$ can be modeled as a time dependent random parameter $\lambda(t, \omega)$ in the equation of state 
$$p = \rho e (\lambda-1) = A\rho^\lambda.$$  Problems with uncertain $\lambda$ were considered in, for example, \cite{tryoen2010}.

We impose the initial condition
\begin{equation}
    \rho_0(x)=2-\phi(x), \quad p_0(x)=A_{0} \rho(x)^{\lambda_0}, \quad \lambda_0 = \frac{4}{3}, \quad A_0 = 1,
\end{equation}
and the Neumann boundary condition on $x\in[0,2]$.

To be more precise, we let
\begin{equation}
    \lambda = \bar{\lambda} + \omega_k \quad  \text{ for } \quad \frac{k-1}{K}T \leq t < \frac{k}{K}T, \quad k=1,\dots,K,
\end{equation}
with $\bar{\lambda}=\frac{4}{3}$, $\omega_k\sim\mathcal{U}(-\frac{1}{3}, \frac{1}{3})$.

We use a well-balanced scheme introduced by K{\"a}ppeli and Mishra \cite{kappeli2014wellbalanced}, where the hydrostatic steady state of \eqref{eqn:Euler} is
\begin{equation}
    v\equiv 0, \quad p_x = -\rho\phi_x, \quad h+\phi = \text{const}, \quad h = e+\frac{p}{\rho}.
\end{equation}
The numerical parameters are chosen as follows
\begin{equation}\label{eqn:params2}
    \kappa = \frac{\Delta t_l}{\Delta x_l} = 0.1,\quad \Gamma =  2,\quad \Delta x_0 = 0.25,\quad T = 2, \quad N_{I} = 500.
\end{equation}

Here we show the numerical results for Case I: $\lambda$ with fixed $K$, and Case II: white noise $\lambda$. 

\subsubsection{Case I}
\cref{fig:Euler1} ($K=1$) and \cref{fig:Euler2} ($K=8$) show that $\Var\left[P_l^{(l)}\right]\approx O(1)$ and $\Var\left[P_l^{(l)}-P_{l}^{(l-1)}\right]\approx O(\left(\Delta t\right)^{1.5})$, i.e. $\beta_0=0$, $\beta=1.5$. The costs of MC and MLMC are $O(\delta^{-4})$ and $O(\delta^{-2.5})$ respectively. Both results are consistent with the analysis.

\begin{enumerate}

    \item $K=1$
    \begin{figure}[H]
      \subfigure{
        \includegraphics[width=0.45\textwidth]{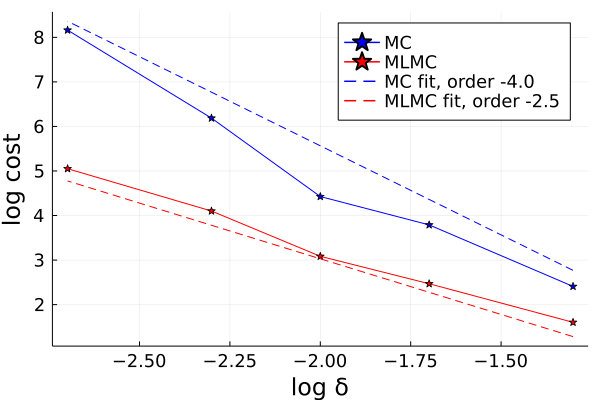}}
      \subfigure{
        \includegraphics[width=0.45\textwidth]{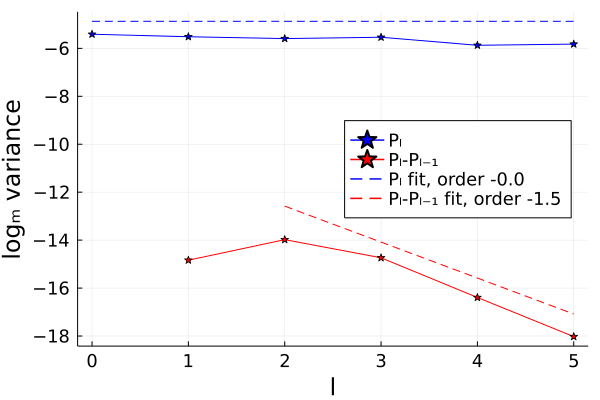}}
      \caption{Example 2: Case I, $K=1$, the computational cost w.r.t accuracy threshold $\delta$ (left) and variances on different levels (right).}\label{fig:Euler1}
    \end{figure}
    
    \item $K=8$
    \begin{figure}[H]
      \subfigure{
        \includegraphics[width=0.45\textwidth]{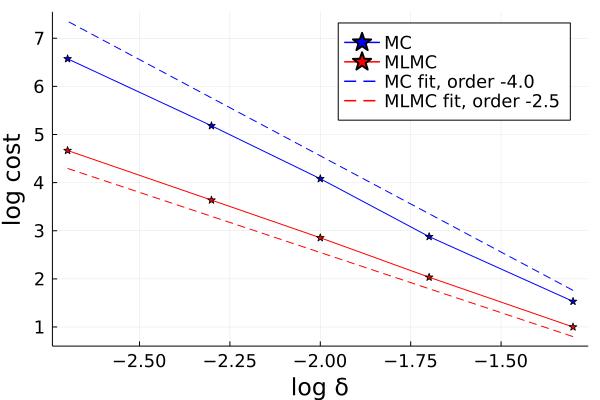}}
      \subfigure{
        \includegraphics[width=0.45\textwidth]{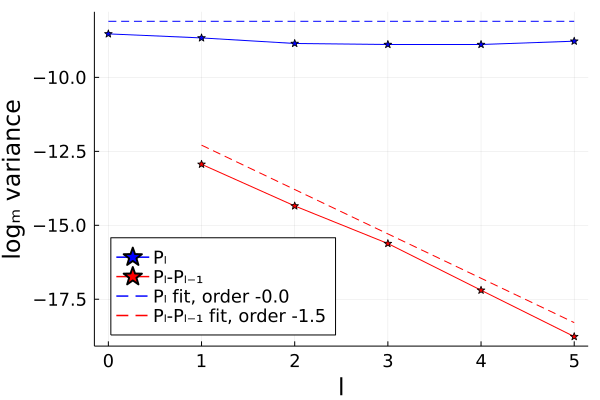}}
      \caption{Example 2: Case I, $K=8$, the computational cost w.r.t accuracy threshold $\delta$ (left) and variance on different levels (right).}\label{fig:Euler2}
    \end{figure}
    
\end{enumerate}

\subsubsection{Case II}
In this case, we generate white noise $\lambda$ with $K_l = T/\Delta t_l$ on the $l$th level, and use similar trick as \eqref{eqn:combinetrick} to construct coarser random variables.  

From \cref{fig:Euler3}, we observe that $\Var\left[P_l^{(l)}\right]$ and $\Var\left[P_l^{(l)}-P_{l}^{(l-1)}\right]$ are both $O(\Delta t_l)$, i.e. $\beta_0=\beta=1$. Comparing with the variances in \cref{fig:Euler1}, \cref{fig:Euler2}, it is evident that the temporal white noise  leads to the decay of $\Var\left[P_l^{(l)}\right]$. The costs of MC and MLMC are both $O(\delta^{-3})$.

    \begin{figure}[H]
  \subfigure{
    \includegraphics[width=0.45\textwidth]{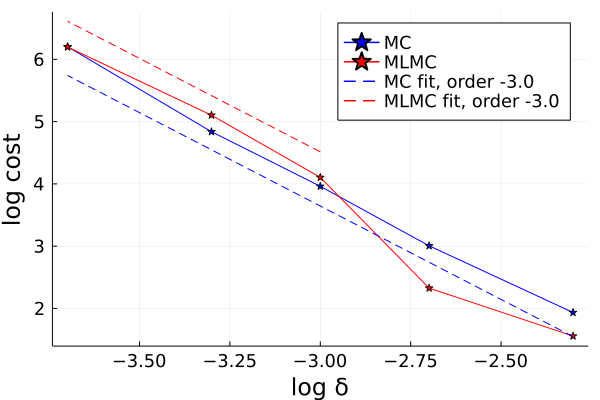}}
  \subfigure{
    \includegraphics[width=0.45\textwidth]{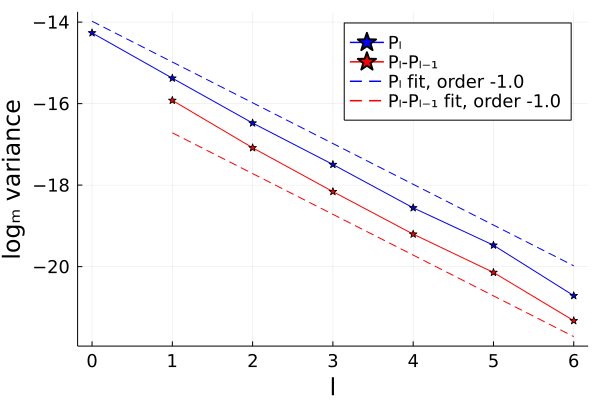}}
  \caption{Example 2: Case II, the computational cost w.r.t accuracy threshold $\delta$ (left) and variance on different levels (right).}\label{fig:Euler3}
\end{figure}

\subsection{Shallow water equations with random tomography $B(x, \omega)$}
\label{sec:comparison:shallowwater}

In this example, we consider space dependent random parameters instead of time dependent random parameters, using the  shallow water equations
\begin{equation}
    \left\{\begin{aligned}
    & h_t + (hu)_x = 0, \\
    & (hu)_t + \left(hu^2 + \frac{gh^2}{2}\right)_x = -ghB_x.
    \end{aligned}\right.
\end{equation}
Here $h(t,x)$ is the height, $u(t,x)$ is the mean velocity, with the gravitational acceleration constant $g$, and the bottom topography $B(x)$. In view of the measurement error \cite{liu2010}, we can introduce the following random tomography $B(x, \omega)$ defined via $K$ random variables, such that
\begin{equation}
    B(x, \omega) = 1 + \omega_k\sin(\pi x) \quad  \text{ for } \quad \frac{k-1}{K} \leq x < \frac{k}{K}, \quad k=1,\dots,K,
\end{equation}
with $\omega_k\sim\mathcal{U}(0,1)$ and the steady state
\begin{equation}
    q:=hu\equiv \widehat{q} = \text{const},\quad E:=\frac{u^2}{2}+g(h+B)\equiv \widehat{E} = \text{const}.
\end{equation}

We treat the still water case with $q \equiv 0$ ($v \equiv 0$) using the well-balanced central-upwind scheme \cite{kurganov2018finite}. The initial conditions 
\begin{equation}
    h(x,0)=5+e^{\cos(2\pi x)}, \quad u(x,0)=\sin(\cos(2\pi x)),
\end{equation}
and the periodic boundary condition on $x\in[0,1]$ are imposed.

We use the following numerical parameters
\begin{equation}\label{eqn:params1}
    \kappa = \frac{\Delta t}{\Delta x} = 0.05,\quad \Gamma =  2,\quad \Delta x_0 = \frac{1}{64}, \quad T = 0.1, \quad N_{I} = 500.
\end{equation}

Similarly as before, we have two qualitatively different cases: Case I: $B(x,\omega)$ with fixed $K$, and Case II: spatial white noise $B(x, \omega)$. 

\subsubsection{Case I}

We use $K=1$ in \cref{fig:still1} and $K=8$ in \cref{fig:still2}. We observe that $\Var\left[P_l^{(l)}\right]=O(1)$ ($\beta_0=0$), the numerical values for the decay orders of $\Var\left[P_l^{(l)}-P_{l}^{(l-1)}\right]$ in these two experiments are $\beta=2.5\, (K=1)$ and $\beta=1.0\, (K=8)$ respectively. 

The cost of MC is approximately $O(\delta^{-4})$ since $\beta_0=0$. 
In the case $K=1$, since $\beta = 2.5 > \gamma$, the cost of MLMC is approximately $O(\delta^{-2})$. While in the case $K=8$, the cost of MLMC is approximately $O(\delta^{-3})$. As we expected, MLMC is better than MC method by one or two orders of magnitude.

\begin{enumerate}
    \item $K=1$
    \begin{figure}[H]
      \subfigure{
        \includegraphics[width=0.45\textwidth]{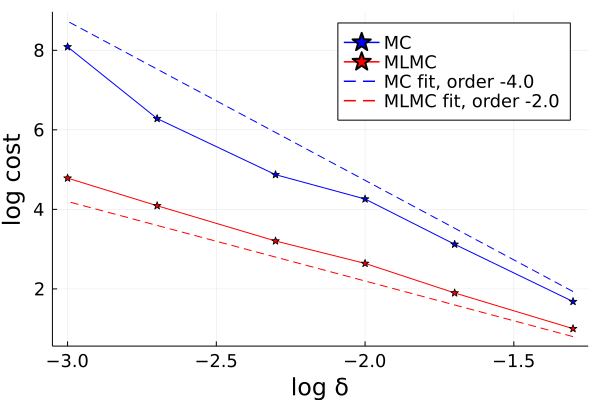}}
      \subfigure{
        \includegraphics[width=0.45\textwidth]{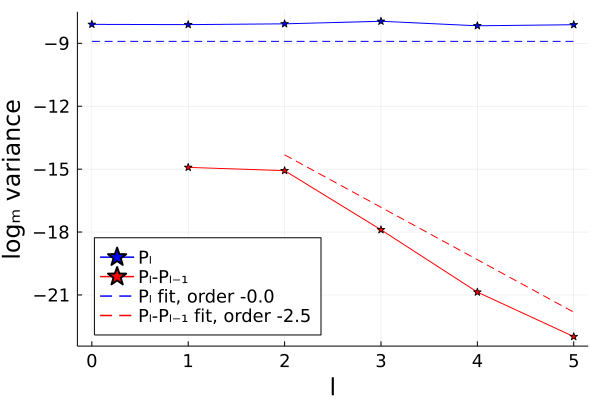}}
      \caption{Example 3: Case I, $K=1$, the computational cost w.r.t accuracy threshold $\delta$ (left) and variance on different levels (right).}\label{fig:still1}
    \end{figure}

    \item $K=8$
    \begin{figure}[H]
      \subfigure{
        \includegraphics[width=0.45\textwidth]{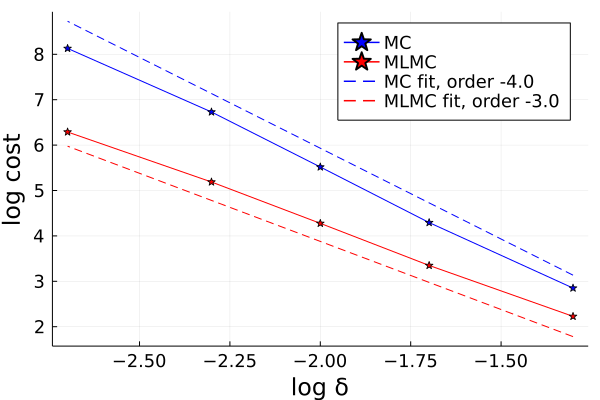}}
      \subfigure{
        \includegraphics[width=0.45\textwidth]{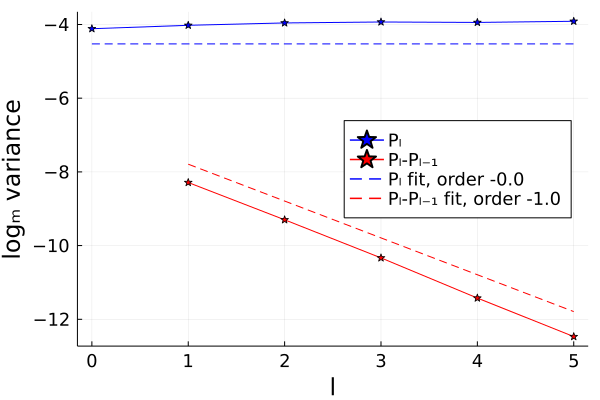}}
      \caption{Example 3: Case I, $K=8$, the computational cost w.r.t accuracy threshold $\delta$ (left) and variance on different levels (right).}\label{fig:still2}
    \end{figure}

\end{enumerate}

\subsubsection{Case II}

When $B(x,\omega)$ is modeled by spatially dependent white noise, we have $K_l=X_l$ for the $l$th level. When a sample $P_l^{(l)}$ is computed on level $l$, the value of $B(x_i, \omega)$ is generated by 
\begin{equation}
    B(x_{i-1}, \omega) = 1 + \omega_i\sin(\pi x_{i-1}), \quad \omega_i\sim\mathcal{U}(0,1), \quad i = 1,\dots, X_{l}.
\end{equation}
To calculate $P_l^{(l-1)}$, we first construct $\left\{\tilde{\omega}_{i}\right\}_{i=1}^{X_{l-1}}$ using \eqref{eqn:combinetrick},
then let
\begin{equation}
    \tilde{B}(x_{(i-1)\Gamma}, \omega) = 1 + \tilde{\omega}_{i}\sin(\pi x_{(i-1)\Gamma}), \quad i = 1,\dots,X_{l-1}(=X_l/\Gamma).
\end{equation}

\cref{fig:still3} shows that the decays of $\Var\left[P_l^{(l)}\right]$ and $\Var\left[P_l^{(l)}-P_{l}^{(l-1)}\right]$ are very slow, namely, we have $\beta\simeq 0$ and $\beta_0\simeq 0$. In this case, both MC and MLMC methods are expensive. In fact, MLMC has a higher cost compared to MC.

    \begin{figure}[H]
      \subfigure{
        \includegraphics[width=0.45\textwidth]{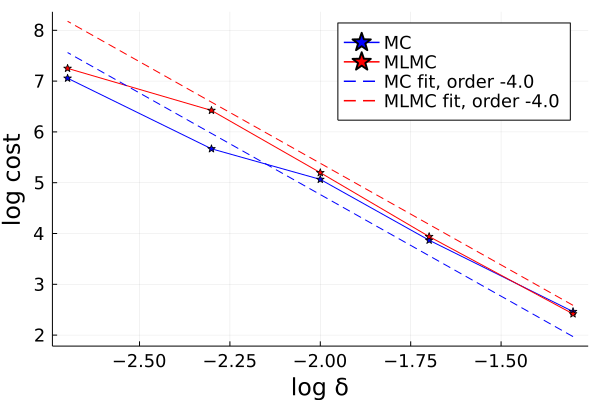}}
      \subfigure{
        \includegraphics[width=0.45\textwidth]{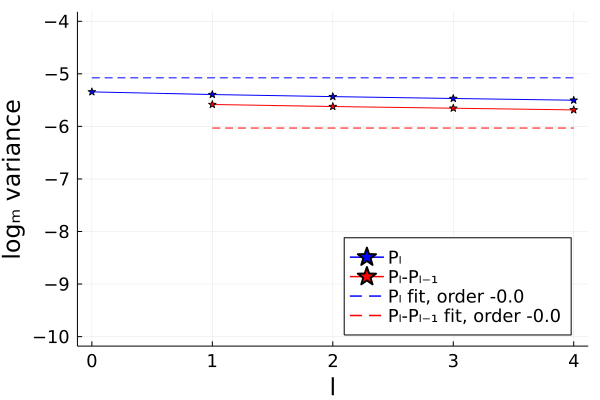}}
  \caption{Example 3: Case II, the computational cost w.r.t accuracy threshold $\delta$ (left) and variance on different levels (right).}\label{fig:still3}
\end{figure}

\section{Comparing MLMC and MC for the linear Jin-Xin relaxation model with random choice method}
\label{sec:comparison:randomchoice}

While the first three examples in \S~\ref{sec:comparison:scalar} and \S~\ref{sec:comparison:eulerandshallow} are hyperbolic equations with random parameters, the last example is of different nature. In this example, we apply a randomized algorithm to a multiscale deterministic equation--the Jin-Xin hyperbolic relaxation model \cite{jin1995relaxation}. This algorithm is in the spirit of the Glimm's random choice method \cite{Glimm1965solutions}, and the random variables are in both spatial and temporal dimensions. We identify similar behavior of MLMC in this example as the third example. This simple model can shed some light on the application of randomized algorithms to more complicated problems.

In particular, we deal with the following one-dimensional linear Jin-Xin relaxation model \cite{jin1995relaxation} 
\begin{equation}\label{mod:origin}
 \left\{\begin{aligned}
 \frac{\partial u}{\partial t}+\frac{\partial v}{\partial x} &=0, \\
 \frac{\partial v}{\partial t}+a \frac{\partial u}{\partial x} &=-\frac{1}{\varepsilon}(v-b u),
 \end{aligned}\right.
\end{equation}
 where $(x, t) \in [-1,1] \times [0,T]$ with the periodic boundary condition, and the initial data
$ u(x, 0)=u_{0}(x)$, $\quad v(x, 0)=v_{0}(x)$. $a>0$, $b\in\mathbb{R}$, and  $\varepsilon>0$ is the relaxation parameter. Asymptotic analysis \cite{jin1995relaxation, Natalini2001recent} shows that when $\varepsilon$ is small, the stability condition $a\geq b^2$ ensures the well-posedness of model. In the limit $\varepsilon \rightarrow 0$, the second equation of \eqref{mod:origin} relaxes to the local equilibrium  $ v = bu$,
and the first equation of \eqref{mod:origin} reduces to the scalar advection equation
$   u_t + b u_x =0.
$ 
The Jin-Xin model is a prototypical multiscale hyperbolic equation, and can be treated by the so-called asymptotic preserving (AP) method \cite{jin2000uniformly}. 

In the following, we will develop the asymptotic preserving Monte Carlo method (APMC) for the equation \eqref{mod:origin}, by combining the Glimm's random choice method \cite{Glimm1965solutions} and the asymptotic preserving principle.
A random algorithm is interesting not only for theoretical interest \cite{Glimm1965solutions} but also of practical interest \cite{BaoJin, chorin-RCM, Dim-Par}.  We will compare the APMC method with its multilevel version (MLAPMC). We will demonstrate that, since the random variables in the APMC method involves both spatial and temporal dimensions, the cost of MLAPMC is of the same order as APMC. 

\subsection{Asymptotic preserving Monte Carlo method (APMC)}
\label{sec:comparison:randomchoice:apmc}

The equation \eqref{mod:origin} is stiff in the sense that the time step of any explicit method has to satisfy $\Delta t = O(\varepsilon)$ by the stability requirement, which is prohibitively expensive as $\varepsilon\rightarrow 0$. The asymptotic preserving (AP) method  \cite{jin1999efficient} allows to choose $\Delta t$ independent of $\varepsilon$, while it still attains the correct asymptotic limit. To be precise, \eqref{mod:origin} can be split into two steps,
\begin{equation}\label{mod:sub-con}
 \text{(convection step)}\quad
\left\{\begin{aligned}
 \frac{\partial u}{\partial t}+\frac{\partial v}{\partial x}&=0, \\
 \frac{\partial v}{\partial t}+a \frac{\partial u}{\partial x}&=0,
 \end{aligned}\right.
\end{equation}
 which is a linear convection equation, and \begin{equation}\label{mod:sub-rel}
\text{(relaxation step)}\quad
\left\{\begin{aligned}
 \frac{\partial u}{\partial t}&=0, \\
 \frac{\partial v}{\partial t}&=-\frac{1}{\varepsilon}(v-b u),
 \end{aligned}\right.
\end{equation}
which is a relaxation equation with a stiff source only in the equation for $v$. 

We can construct a random choice method for the split system, 
\begin{itemize}
    \item convection step: 
     \eqref{mod:sub-con} can be transformed to a diagonal form as
    \begin{equation}\label{con:char-mod}
     \left\{\begin{aligned}
     \frac{\partial \rho}{\partial t}+\sqrt{a} \frac{\partial \rho}{\partial x}=0, \\
     \frac{\partial l}{\partial t}-\sqrt{a} \frac{\partial l}{\partial x}=0,
     \end{aligned}\right.
    \end{equation}
    by introducing the characteristic variables \begin{equation}
    \rho=\sqrt{a} u+v, \quad l=\sqrt{a} u-v.
    \label{eqn:char}
    \end{equation}
    The explicit upwind scheme to  \eqref{con:char-mod} writes
     \begin{equation}\label{con:char-upw}
      \begin{aligned}
     \frac{\rho_{i}^{n+\frac{1}{2}}-\rho_{i}^{n}}{\Delta t} &= -\sqrt{a} \frac{\rho_{i}^{n}-\rho_{i-1}^{n}}{\Delta x},\\
     \frac{l_{i}^{n+\frac{1}{2}}-l_{i}^{n}}{\Delta t}       &=\sqrt{a}\frac{l_{i+1}^{n}-l_{i}^{n}}{\Delta x},
      \end{aligned} 
     \end{equation}
     namely,
     \begin{equation}\label{con:char-convex}
      \begin{aligned}
     \rho_{i}^{n+\frac{1}{2}}&=(1-\nu) \rho_{i}^{n}+\nu \rho_{i-1}^{n},\\
     l_{i}^{n+\frac{1}{2}}&=(1-\nu) l_{i}^{n}+\nu l_{i+1}^{n},
      \end{aligned}
     \end{equation}
     where $\displaystyle\nu=\sqrt{a} \frac{\Delta t}{\Delta x}$ is the CFL number.     $u_i^{n+\frac{1}{2}}$ and $v_i^{n+\frac{1}{2}}$ can be obtained by
     \begin{equation}\label{con:char-2uv}
     \begin{aligned}
     u_{i}^{n+\frac{1}{2}} &=\frac{\rho_{i}^{n+\frac{1}{2}}+l_{i}^{n+\frac{1}{2}}}{2 \sqrt{a}}, \\
     v_{i}^{n+\frac{1}{2}} &=\frac{\rho_{i}^{n+\frac{1}{2}}-l_{i}^{n+\frac{1}{2}}}{2}.
     \end{aligned}
    \end{equation}
     
    The Monte Carlo method can be applied to update $\rho$ and $l$, since \eqref{con:char-convex} are convex combinations. For example, we generate samples according to the following distribution
    \begin{equation}\label{con:char-sto}
    \begin{aligned}
      \rho_i^{n+\frac{1}{2}} &= \left\{\begin{array}{ll}
          {\rho_i^n,}&{\text{with probability  } 1-\nu},\\
          {\rho_{i-1}^n,}&{\text{with probability  } \nu},
      \end{array}
      \right.\\
      l_i^{n+\frac{1}{2}} &= \left\{\begin{array}{ll}
          {l_i^n,}&{\text{with probability  } 1-\nu},\\
          {l_{i+1}^n.}&{\text{with probability  } \nu}.
      \end{array}
      \right.
    \end{aligned}
    \end{equation}
    In the implementation, we let $\rho_{i}^{n+\frac{1}{2}} = \rho_{i}^{n}$ and sample a random variable $\xi_{i}^{n} \sim \mathcal{U}(0,1)$. If $\xi_{i}^{n} <\nu$, we update $\rho_{i}^{n+\frac{1}{2}} = \rho_{i-1}^{n}$. 

    \item relaxation step: we use the fully implicit Euler's method to treat the stiff relaxation equation \eqref{mod:sub-rel},
    \begin{equation}\label{rel:back-EL}
    \begin{aligned}
    \frac{u_{i}^{n+1}-u_{i}^{n+\frac{1}{2}}}{\Delta t}&=0, \\
    \frac{v_{i}^{n+1}-v_{i}^{n+\frac{1}{2}}}{\Delta t}&=-\frac{1}{\varepsilon}\left(v_{i}^{n+1}-b u_{i}^{n+1}\right).
    \end{aligned}
    \end{equation}
    Thus
    \begin{equation}\label{rel:convex}
      \begin{aligned}
        u_{i}^{n+1}&=u_{i}^{n+\frac{1}{2}},\\
        v_{i}^{n+1}&=\frac{\varepsilon}{\varepsilon+\Delta t} v_{i}^{n+\frac{1}{2}}+\frac{\Delta t}{\varepsilon+\Delta t} b u_{i}^{n+\frac{1}{2}}.
      \end{aligned}
    \end{equation}
    $v_{i}^{n+1}$ is a convex combination of $v_{i}^{n+\frac{1}{2}}$ and $b u_{i}^{n+\frac{1}{2}}$, which has the following Monte Carlo approximation
    \begin{equation}\label{rel:uv-sto}
      v_i^{n+1} = \left\{\begin{array}{ll}
          {v_i^{n+\frac{1}{2}},}&{\text{with probability  } \displaystyle p,}\\
          {bu_{i}^{n+\frac{1}{2}}.}&{\text{with probability  } 1-\displaystyle p,}
      \end{array}
      \right.
    \end{equation}
    with $\displaystyle p=\frac{\varepsilon}{\varepsilon+\Delta t}$. 

\end{itemize}

The APMC algorithm can be summarized as follows
\begin{algorithm}[H]
  \caption{Asymptotic Preserving Monte Carlo (APMC) - fully-random version}\label{alg:apmc}
  Given the number of samples $N\in \mathbb{N} $, $T>0$, time step $\Delta t$, mesh size $\Delta x$ (or equivalently, $\nu = \sqrt{a}\Delta t /\Delta x$). Generating $N$ samples $\left\{u(k),v(k)\right\}_{k=1}^N$.
  \begin{algorithmic}[1]
  \STATE\label{Step 1} Initialize $u^0_i(k) = u_0(x_i), v^0_i(k) = v_0(x_i)$, $k = 1,\dots, N$;
  
  \FOR{$n = 0, \dots, T/\Delta t-1$}
  \FOR{$k=1,\dots, N$}
  
  \STATE Compute $\rho^{n}_i(k), l^{n}_i(k)$ as in \eqref{eqn:char};
  \STATE\label{alg:apmc-con} Sample $\rho^{n+\frac{1}{2}}_i(k)$ and $l^{n+\frac12}_i(k)$ as in \eqref{con:char-sto};
  \STATE Compute $u^{n+\frac{1}{2}}_i(k), v^{n+\frac{1}{2}}_i(k)$ as in \eqref{con:char-2uv};
  \STATE $u^{n+1}_i(k) = u^{n+\frac{1}{2}}_i(k)$;
  \STATE\label{alg:apmc-rex} Sample $v^{n+1}_i(k) = v^{n+\frac{1}{2}}_i(k)$ as in \eqref{rel:uv-sto};

  \ENDFOR
  \STATE $\hat{u}_i^n = \frac{1}{N}\sum_{k=1}^N u_i^n(k)$, $\hat{v}_i^n = \frac{1}{N}\sum_{k=1}^N v_i^n(k)$;

  \ENDFOR

  \end{algorithmic}
\end{algorithm}


We can replace the convection Step \ref{alg:apmc-con}
in \cref{alg:apmc} by the deterministic scheme \eqref{con:char-convex}. The resulting algorithm is called the \emph{semi-random} version of APMC. Also, we can obtain the deterministic method if we replace the Step \ref{alg:apmc-con} by \eqref{con:char-convex} and the Step \ref{alg:apmc-rex} by \eqref{rel:back-EL}, respectively.

\subsection{Variance analysis}
\label{sec:comparison:randomchoice:analysis}

The deterministic method is asymptotic preserving since it has the correct $\varepsilon \to 0$ limit, as proved in Appendix \ref{sec:appendix:numericallimit}. To apply the Monte Carlo method as $\varepsilon \rightarrow 0$, the variance should be uniformly bounded, i.e.
\begin{equation}\label{varAna:uni-bd}
  \Var\left[u_i^n\right],\Var\left[v_i^n\right] \leq C,
\end{equation}
where $C$ is a constant independent of $\varepsilon$. 

We present the following lemmas for the expectations and variances. The proofs are given in Appendix \ref{sec:appendix:varianceanalysis}.


\begin{lemma}\label{thm:det-energy}
  Under the condition $a\geq b^2$, the energy
  \begin{equation}\label{det:prop-energy}
    \begin{aligned}
      \mathcal{E}_n := \sum_i\left[(a-b^2)\left(\E\left[u_i^{n}\right]\right)^2+\left(\E\left[bu_i^{n}-v_i^{n}\right]\right)^2\right]\Delta x,\\
    \end{aligned}
  \end{equation}
  is monotonically decreasing,
  \begin{equation}\label{det:prop-enerdecay}
    \mathcal{E}_n\leq \mathcal{E}_{n-1}\leq \cdots \leq \mathcal{E}_0.
  \end{equation}
  Moreover, the $l_2$ norms of $\E\left[u^n\right],\E\left[v^n\right]$ remain bounded in the sense that
  \begin{equation}\label{det:prop-ubound}
    \left\|\E\left[u^n\right]\right\|_{2} = \sqrt{\sum_i\left(\E\left[u_i^n\right]\right)^2\Delta x}\leq \sqrt{\frac{\mathcal{E}_0}{a-b^2}},
  \end{equation}
  \begin{equation}\label{det:prop-vbound}
    \left\|\E\left[v^n\right]\right\|_{2} = \sqrt{\sum_i\left(\E\left[v_i^n\right]\right)^2\Delta x}\leq \sqrt{\frac{2a\mathcal{E}_0}{a-b^2}}.
  \end{equation}
\end{lemma}

\begin{remark}
 We have the following estimate for $\E\left[u_{i+1}^n-u_{i}^n\right]$ and $\E\left[v_{i+1}^n-v_{i}^n\right]$,
\begin{equation}\label{rem:square-variation}
  \begin{aligned}
    &\quad\sum_i\left[(a-b^2)\left(\E\left[u_{i+1}^{n+1}-u_{i}^{n+1}\right]\right)^2+\left(\E\left[b(u_{i+1}^{n+1}-u_i^{n+1})-(v_{i+1}^{n+1}-v_i^{n+1})\right]\right)^2\right]\Delta x\\
    &\leq \sum_i\left[(a-b^2)\left(u_{i+1}^0-u_i^{0}\right)^2+\left(b\left(u_{i+1}^0-u_i^{0}\right)-\left(v_{i+1}^0-v_i^{0}\right)\right)^2\right]\Delta x:= \mathcal{C}_0(\Delta x)^2.
  \end{aligned}
\end{equation}

Similar as \eqref{det:prop-ubound}, \eqref{det:prop-vbound}, we have \begin{equation}
    \sqrt{\sum_{i} \left(\E\left[\frac{u_{i+1}^{n+1}-u_{i}^{n+1}}{\Delta x}\right]\right)^2 \Delta x} \leq \sqrt{\frac{\mathcal{C}_0}{a-b^2}}, \quad 
    \sqrt{\sum_{i} \left(\E\left[\frac{v_{i+1}^{n+1}-v_{i}^{n+1}}{\Delta x}\right]\right)^2 \Delta x} \leq \sqrt{\frac{2a\mathcal{C}_0}{a-b^2}}.
\end{equation}
\end{remark}

\begin{lemma}\label{prop:det-buv}
 As $\varepsilon\rightarrow 0^+$, $\E\left[bu_i^n-v_i^n\right]\to 0$, and for any fixed $\varepsilon>0$, $\sum_j\sum_i\left(\E\left[bu_i^j-v_i^j\right]\right)^2\Delta x$ is uniformly bounded in the sense that
 \begin{equation}\label{det:prop-sum-vbu}
  \sum_{j=1}^\infty \sum_i \left(\E\left[bu_i^j-v_i^j\right]\right)^2\Delta x \leq \frac{\displaystyle p^2}{1-\displaystyle p^2}\mathcal{E}_0.
 \end{equation}
 In particular, $\sum_i \left(\E\left[bu_i^n-v_i^n\right]\right)^2\Delta x\to 0$ as $n\rightarrow\infty$.
\end{lemma}

To better understand the variances introduced in the convection step and in the relaxation step, respectively, we first consider the semi-random APMC with the random  relaxation step and the deterministic convection step. We have the following lemma for variances. 
\begin{lemma}\label{thm:s-sto-energy}
  The variance of the semi-random APMC solution is bounded in the sense that 
  \begin{equation}\label{s-sto:prop-var-bd}
    \sum_i \left[(a-b^2)\Var\left[\tilde{u}_i^{n}\right]+\Var\left[b\tilde{u}_i^{n}-\tilde{v}_i^{n}\right]\right] \Delta x \leq \frac{\displaystyle p}{1+\displaystyle p} \mathcal{E}_0, \quad \forall n,
  \end{equation}
  where $\mathcal{E}_0 = \sum_i\left[(a-b^2)\left(u_i^{0}\right)^2+\left(bu_i^{0}-v_i^{0}\right)^2\right] \Delta x$. Namely, for any $\varepsilon\ll 1$, the variance is uniformly bounded and tends to zero as $\varepsilon\rightarrow 0$.
\end{lemma}

\begin{remark}
  We can use the same trick as in the proof of \cref{prop:det-buv} to get an estimate of $\Var\left[bu_i^n-v_i^n\right]$,

  \begin{equation}\label{s-sto:sum-var-buv}
    \begin{aligned}
      &\sum_{j=1}^{n+1} \sum_i \Var\left[bu_i^{j}-v_i^{j}\right] \Delta x 
    \leq \frac{1}{\displaystyle p} \sum_{j=1}^{n+1} \sum_i\left(\E\left[bu_i^{j}-v_i^{j}\right]\right)^2 \Delta x \leq \frac{\displaystyle p}{1-\displaystyle p^2} \mathcal{E}_0, \quad \forall n,
    \end{aligned}
  \end{equation}
  in particular,
  \begin{equation}\label{s-sto:var-buv-0}
    \sum_i \Var\left[bu_i^{n}-v_i^{n}\right] \Delta x \rightarrow 0, \quad \text{as} \quad n \rightarrow \infty .
  \end{equation}
\end{remark}

The main difference for the variance analysis of the fully-random APMC in the following \cref{thm:f-sto-energy}, is to use the conditional variance formula for characteristic variables in the random convection step. 
\begin{lemma}\label{thm:f-sto-energy}
  The variance of the fully-random APMC solutions is bounded in the sense that 
  \begin{equation}\label{f-sto:prop-var-bd}
    \sum_i \left[(a-b^2)\Var\left[u_i^{n}\right]+\Var\left[bu_i^{n}-v_i^{n}\right]\right] \Delta x \leq \frac{\displaystyle p}{1+\displaystyle p} \mathcal{E}_0+\frac{a(1-\nu)}{\nu}T\mathcal{C}_0\Delta t,\quad \forall n,
  \end{equation}
  where
\begin{equation*}
  \begin{aligned}
    \mathcal{E}_0 & = \sum_i\left[(a-b^2)\left(u_i^{0}\right)^2+\left(bu_i^{0}-v_i^{0}\right)^2\right] \Delta x,\\
    \mathcal{C}_0 & = \sum_i\left[(a-b^2)\left(\frac{u_{i+1}^0-u_i^{0}}{\Delta x}\right)^2+\left[b\left(\frac{u_{i+1}^0-u_i^{0}}{\Delta x}\right)-\left(\frac{v_{i+1}^0-v_i^{0}}{\Delta x}\right)\right]^2\right] \Delta x.
  \end{aligned}
\end{equation*}
Then for any $\varepsilon\ll 1$, the variance is uniformly bounded and tends to zero as $\varepsilon\rightarrow 0$, $\Delta t\rightarrow 0$. Therefore, the APMC method  still works in the regime when $\varepsilon \ll 1$ and $\Delta t$ independent of $\varepsilon$, i.e. has the asymptotic preserving property.
\end{lemma}

\subsection{Multilevel APMC}
We introduce the multilevel version of APMC (MLAPMC) algorithm here. Since the random variables are distinct at the finest temporal and spatial discretization level, we encounter similar situation as the Case II in Example \ref{sec:comparison:scalar}, Example \ref{sec:comparison:euler} and Example \ref{sec:comparison:shallowwater}. Therefore, we need to use the trick introduced in \cite{lovbak2021multilevel} and in \eqref{eqn:combinetrick} to construct coarser random variables and compute MLMC corrections.  

The fully-random MLAPMC algorithm of one sample on level $l$ can be stated as follows.
\begin{algorithm}[H]
  \caption{Multilevel APMC (MLAPMC) - fully-random version}\label{alg:mlapmc}
  Given $T>0$, time step $\Delta t_l$, CFL number $\nu = \sqrt{a}\Delta t_l /\Delta x_l$.
  \begin{algorithmic}[1]
  \STATE Initialize $u^0_i = u_0(x_i), v^0_i = v_0(x_i)$, $\bar{u}^0_{i_1} = u_0(x_{i_1\Gamma }), \bar{v}^0_{i_1} = v_0(x_{i_1\Gamma })$, for $i_1 =0, \dots,X_{l-1}-1$;
  
  \FOR{$n = 0, \dots, M_{l-1}-1$}
  \FOR{$k=0, \dots,\Gamma-1$}
  \STATE compute $\rho^{n^\prime}_i, l^{n^\prime}_i$ as in \eqref{eqn:char}, where $n^\prime=n \Gamma +k$;
  \STATE sample $\rho^{n^\prime+\frac{1}{2}}_i$ with $\xi_{i,k} \sim \mathcal{U}(0,1)$ and $l^{n^\prime+\frac{1}{2}}_i$ with $\eta_{i,k} \sim \mathcal{U}(0,1)$ as in \eqref{con:char-sto};
  \STATE compute $u^{n^\prime+\frac{1}{2}}_i$, $v^{n^\prime+\frac{1}{2}}_i$ as in \eqref{con:char-2uv};
  \STATE $u^{n^\prime+1}_i = u^{n^\prime+\frac{1}{2}}_i$;
  \STATE sample $v^{n^\prime+1}_i$ with $\zeta_{i,k} \sim \mathcal{U}(0,1)$ as in \eqref{rel:uv-sto};
  \ENDFOR
  \STATE compute $\bar{\rho}^{n}_{i_1}, \bar{l}^{n}_{i_1}$ as in \eqref{eqn:char};
  \STATE $\xi_{i_1}=\left(\max_{i_1\leq i<i_1+\Gamma,0\leq k<\Gamma}\{\xi_{i,k}\}\right)^{\Gamma^2}$, $\eta_{i_1}=\left(\max_{i_1\leq i<i_1+\Gamma,0\leq k<\Gamma}\{\eta_{i,k}\}\right)^{\Gamma^2}$;
  \STATE sample $\bar{\rho}^{n+\frac{1}{2}}_{i_1}$ with $\xi_{i_1}$ and $\bar{l}^{n+\frac{1}{2}}_{i_1}$ with $\eta_{i_1}$ as in \eqref{con:char-sto};
  \STATE compute $\bar{u}^{n+\frac{1}{2}}_{i_1}$, $\bar{v}^{n+\frac{1}{2}}_{i_1}$ as in \eqref{con:char-2uv};
  \STATE $\bar{u}^{n+1}_{i_1} = \bar{u}^{n+\frac{1}{2}}_{i_1}$;
  \STATE $\zeta_{i_1}=\left(\max_{i_1\leq i<i_1+\Gamma,0\leq k<\Gamma}\{\zeta_{i,k}\}\right)^{\Gamma^2}$;
  \STATE sample $\bar{v}^{n+1}_{i_1}$ with $\zeta_{i_1}$ as in \eqref{rel:uv-sto};
  \ENDFOR
  \end{algorithmic}
\end{algorithm}

As discussed before, the variance in the asymptotic regime ($\varepsilon \ll 1$) is small, and APMC performs well enough. For comparison, we apply MLAPMC in the non-stiff regime ($\varepsilon = 1$) with semi-random and fully-random schemes respectively. We consider the initial value
\begin{equation}
    u_0(x) =\frac{\sin(x)+1}{2}, \quad v_0(x) = 0,
\end{equation}
with $x\in[-1,1]$, $t\in[0,T]$, and periodic boundary condtion. 

We take the numerical parameters
\begin{equation}
    \nu = \sqrt{a}\frac{\Delta t_l}{\Delta x_l} = \frac{1}{2}, \quad \Gamma = 2, \quad \Delta x_0 = \frac{1}{32}, \quad T = 1, \quad N_I = 500, \quad a = 1, \quad b = 2.
\end{equation}

We implement both MLAPMC and APMC algorithms with accuracy thresholds $\delta = 0.02$, $0.01$, $0.005$, $0.002$, $0.001$. The results are shown in \cref{fig:apmcmlmc} and \cref{fig:apmcfullmlmc}. The costs of two algorithms with respect to $\delta$ are plotted in the left subfigure, and the right subfigure shows the variances of $P_{l}^{(l)}$ and $P_{l}^{(l)}-P_{l}^{(l-1)}$. We can see that both $\Var\left[P_{l}^{(l)}\right]$ and $\Var\left[P_{l}^{(l)}-P_{l}^{(l-1)}\right]$ decay as $\Delta t_l$ decreases in semi-random case, and the decay orders are about $0.5$, i.e. $\beta_0=\beta=0.5$. The costs of APMC and MLAPMC are both $O(\delta^{-3.5})$. While in the fully-random case, $\Var\left[P_{l}^{(l)}\right]$ and $\Var\left[P_{l}^{(l)}-P_{l}^{(l-1)}\right]$ decay very slowly, i.e. $\beta \simeq 0$ and $\beta_0 \simeq 0$. The costs of APMC and MLAPMC in fully-random case are both $O(\delta^{-4})$.

\begin{figure}[H]
      \subfigure{
        \includegraphics[width=0.45\textwidth]{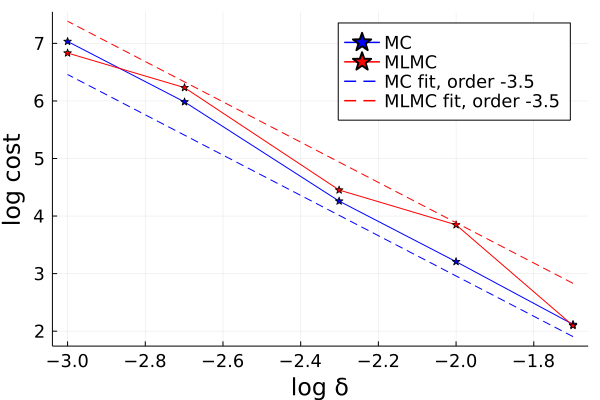}}
      \subfigure{
        \includegraphics[width=0.45\textwidth]{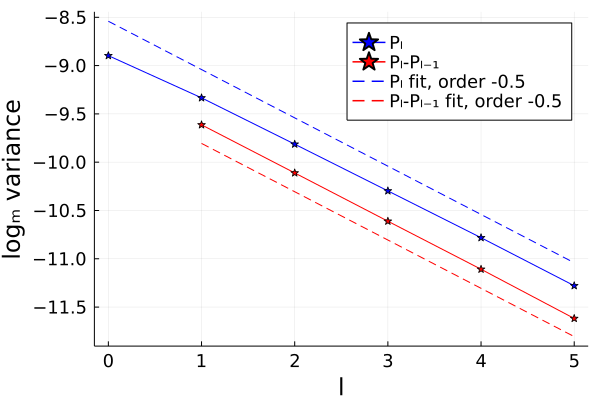}}
  \caption{The computational cost w.r.t accuracy threshold $\delta$ (left) and variance on different levels (right) using semi-random APMC and MLAPMC methods.}
  \label{fig:apmcmlmc}
\end{figure}

\begin{figure}[H]
      \subfigure{
        \includegraphics[width=0.45\textwidth]{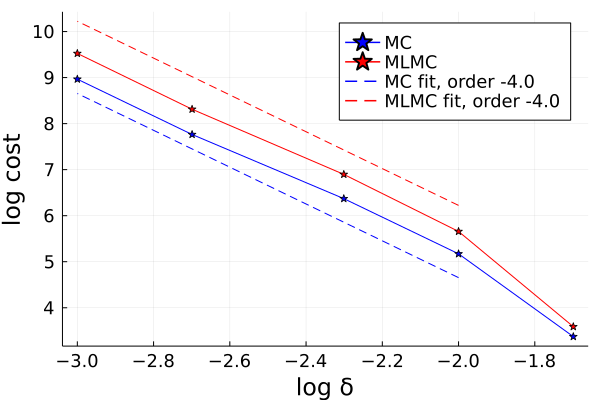}}
      \subfigure{
        \includegraphics[width=0.45\textwidth]{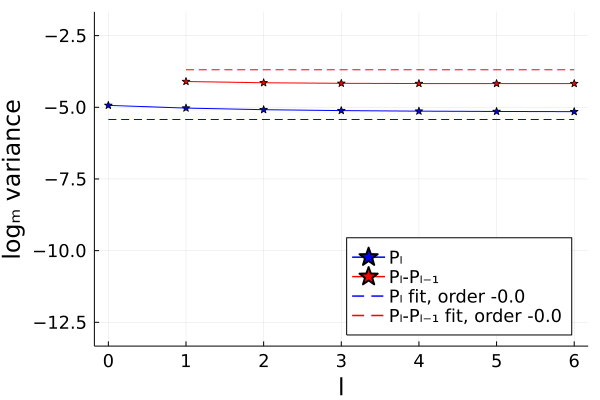}}
  \caption{The computational cost w.r.t accuracy threshold $\delta$ (left) and variance on different levels (right) using fully-random APMC and MLAPMC methods.}
  \label{fig:apmcfullmlmc}
\end{figure}

\section{Conclusion}
\label{sec:conclusion}
In this paper, we study the  Multilevel Monte Carlo method, and compare it with Monte Carlo method for applications in linear and nonlinear hyperbolic systems with randomness originated from either random parameters or randomized algorithms. In many cases, MLMC has lower cost than MC. Nevertheless, if the variance of MC solutions decays at the same rate as MLMC corrections, the standard MC performs well and a multilevel method may not provide improvement.

We conclude through both analytic studies and  a few numerical experiments with different levels of randomness, that the possible cause is white noise like random variables in the SPDE or random algorithm, which may influence the decay rates for MLMC corrections and lead to the deterioration of the MLMC complexity. If the white noise random variables depend only on time, the variances have first order decay with respect to the discretization level or time step, and both MLMC and MC have $O(\delta^{-3})$ computational cost. While if the white noise random variables depend further on space, the variances do not decay and both MLMC and MC have $O(\delta^{-4})$ cost. A rigorous analysis of the scalar advection equation with a random velocity is given and is consistent with our numerical results. 

In summary, we have identified three regimes for the performances of MLMC and MC for hyperbolic equations. More research is needed to investigate SPDEs driven by Gaussian and Levy white noise \cite{Holden1996, Walsh1986}, and to develop acceleration methods for SPDEs and random algorithms for interesting physical applications.

\bibliographystyle{siamplain}
\bibliography{refs.bib}

\appendix


\section{Some known results of MLMC for differential equations}
\label{sec:appendix:mlmcliterature}

We review some applications of MLMC in the literature. For consistency of notations, we always denote $P_l^{(l)}$, $P_l^{(l-1)}$ as the numerical solutions (at final time $T$ if time-dependent) using time step $\Delta t_l$ or $\Delta t_{l-1}$ and mesh size $\Delta x_l$ or $\Delta x_{l-1}$ on level $l$ respectively.
\begin{itemize}
    \item Stochastic differential equations \cite{abdulle2013stabilized,dereich2011multilevel,giles2008multilevel}. 
    
    
    For example, we consider the forward Euler method with time step $\Delta t$ of the SDE 
    $$ \mathrm{d} u(t) = a(u, t) \mathrm{d} t + b(u, t) \mathrm{d} W,$$
    \begin{equation}
        u_{n+1} = u_n + a(u_n,t_n) \Delta t + b(u_n, t_n)\Delta W_n.
    \end{equation}
    
In this case, the random variables are Gaussian white noises. Due to the additivity of Gaussian distribution, it is natural to obtain a random variable on the coarser level by adding up the random variables on the finer level, leading to the reduction of $\Var\left[P_l^{(l)} - P_l^{(l-1)}\right]$. As proved in \cite{giles2008multilevel}, we have $\Var\left[P_{l}^{(l)}\right] = O(1)$, $\Var\left[P_{l}^{(l)} - P_{l}^{(l-1)}\right] = O(\Delta t_l)$, i.e. $\beta_0=0$, $\beta = 1$.

\item PDE with random coefficients \cite{barth2011multi,cliffe2011multilevel}.
    
In subsurface flow applications, the hydraulic conductivity can be modelled as a random field $k =k({\bf{x}}, \omega)$ on $D \times \Omega $. We have the following Darcy's equation with random coefficients in the form
\begin{equation}
        -\nabla \cdot (k({\bf{x}}, \omega)\nabla p({\bf{x}},\omega)) = f({\bf{x}}),\quad \text{in } D.
\end{equation}
To apply MLMC, we need to sample from the input random field $k({\bf{x}}, \omega)$, e.g. by a truncation of the Karhunen-Lo\`{e}ve (KL) expansion \cite{Karhunen1947Uber, Loeve1946Fonctions}. Then for each sample with fixed $\omega$, we solve the PDE with multigrid methods. In this case, the dimension of random variables is fixed by the truncation order of the KL expansion and same random variables are used to calculate $P_{l}^{(l-1)}$, $P_{l}^{(l)}$. The numerical results in \cite{cliffe2011multilevel} show that $\Var\left[P_{l}^{(l)}\right] = O(1)$, $\Var\left[P_{l}^{(l)} - p_{l}^{(l-1)}\right] = O((\Delta t)^2)$, i.e. $\beta_0 = 0$, $\beta = 2$.

\item Particle method for transport equations \cite{lovbak2021multilevel}. 
For the  BGK  model
\begin{equation}
        \partial_t f(x,v,t) +\frac{v}{\varepsilon}\partial_x f(x,v,t) = \frac{1}{\varepsilon^2}\left(\mathcal{M}(v)\rho(x,t)-f(x,v,t)\right),
\end{equation}
where $f(x,v,t)$ is the particle density distribution at time $t$, position $x$ with particle velocity $v$ and $\mathcal{M}(v)$ is the absolute Maxwellian. $\varepsilon$ is the Knudsen number, the dimensionaless mean free path. An asymptotic preserving particle scheme was proposed in \cite{dimarco2018asymptotic}. Given the current state of a particle in the position-velocity phase space, its position is updated based on its velocity in the transport step. Then, the particle changes its velocity with the probability $p_{\Delta t_{l}} = \frac{\Delta t_l}{\varepsilon^2 + \Delta t_l}$ in the collision step, such that the collision is performed if $\xi_l \leq p_{\Delta t_{l}}$ with $\xi_l \sim\mathcal{U}(0,1)$. 

To apply MLMC, L{\o}vbak, Samaey, and Vandewalle \cite{lovbak2021multilevel} proposed to use random variables $\left\{\xi_{l,k}\right\}_{k=1}^{\Gamma}$ in the finer level to construct coarser level random varibales, 
\begin{equation}\label{eqn:maxup}
        \xi_{l-1} = \left(\max_{1\leq k \leq \Gamma} \xi_{l,k}\right)^{\Gamma} \sim \mathcal{U}(0,1).
\end{equation}
If $\xi_{l-1}\leq p_{\Delta t_{l-1}}$, collision is performed on the coarser level. Therefore the trajectories of particles in the finer level and the coarser level are correlated. In this case, particles are independent of each other, similar to the SDE case. $\Var\left[P_{l}^{(l)} - P_{l}^{(l-1)}\right] = O(\Delta t)$, i.e. $\beta = 1$ is proved in \cite[Lemma 9]{lovbak2021multilevel}. The numerical results show that $\Var\left[P_{l}^{(l)}\right] = O(1)$, i.e. $\beta_0 = 0$.

\end{itemize}


\section{Analysis of MLMC and MC in \S~\ref{sec:mlmc:costanalysis}}
\label{sec:appendix:costanalysis}

\subsection{Proof of \cref{thm:costmlmc}, cost of MLMC}
\label{sec:appendix:costanalysis:costMLMC}

\begin{proof}
  Let $L = \left\lceil \frac{\log(\sqrt{2}c_1T^\alpha\delta^{-1})}{\alpha\log \Gamma}\right\rceil$, the bias term $c_1 (\Delta t_L )^\alpha$ is controlled from above and below by
  \begin{equation}
    \frac{1}{\sqrt{2}}\Gamma^{-\alpha}\delta < c_1 (\Delta t_L)^\alpha \leq \frac{1}{\sqrt{2}} \delta,
  \end{equation}
where $\Gamma = \Delta t_{l-1}/\Delta t_l$. Equivalently, we have 
\begin{equation}
  \left|\E\left[\hat{P}^{(L)} - P\right]\right|^2 = \left|\E\left[P^{(L)} - P\right]\right|^2 \leq \frac{1}{2} \delta^2.
\end{equation}
According to \eqref{eqn:KKT}, the variance of the MLMC estimate $\hat{P}^{(L)}$ is
\begin{equation}
  \Var\left[\hat{P}^{(L)}\right] = \sum_{l=0}^L N_l^{-1}V_l = \mu^{-1}\sum_{l=0}^L \sqrt{C_l V_l} = \frac{1}{2}\delta^2, 
  \label{eqn:mlmcvariance}
\end{equation}
which implies 
$\mu = 2\delta^{-2}\sum_{l=0}^L \sqrt{C_l V_l}$, and $N_l = \left\lceil \mu\sqrt{V_l/C_l} \right\rceil$, $l=0, \dots,L$. The total computational cost is
\begin{equation}
  \mathfrak{C} = \sum_{l=0}^L N_l C_l \leq \sum_{l=0}^L \left(\mu\sqrt{C_l V_l}+C_l\right) = 2\delta^{-2}\left(\sum_{l=0}^L \sqrt{C_l V_l}\right)^2 + \sum_{l=0}^L C_l.
\end{equation}
By properties \textit{\textbf{A}}\ref{Assum:1-2}, \textit{\textbf{A}}\ref{Assum:1-3}, we have
\begin{equation}
  \sum_{l=0}^L \sqrt{C_l V_l} \leq \sqrt{c_2c_3} \sum_{l=0}^L (\Delta t_l)^{\frac{\beta-\gamma}{2}} = \left\{ \begin{aligned}  
  & c_5, & \beta > \gamma, \\
  & L+1 \leq c_6 \log \delta^{-1}, & \beta = \gamma,\\
  & c_7 (\Delta t_L)^{\frac{\beta-\gamma}{2}} \leq c_8 \delta^{-(\gamma-\beta)/(2\alpha)}, & 0 < \beta < \gamma,
  \end{aligned}\right.
\end{equation}
and
\begin{equation}
  \sum_{l=0}^L C_l \leq c_3 \sum_{l=0}^L (\Delta t_l)^{-\gamma} \leq c_9 \delta^{-\gamma/\alpha} \leq c_9\delta^{-2}.
\end{equation}
In summary, the computational cost $\mathfrak{C}$ is bounded by
\begin{equation}
  \begin{aligned}
  \mathfrak{C} &\leq \left\{ \begin{aligned}  
  & c_4\delta^{-2}, & \beta > \gamma, \\
  & c_4\delta^{-2}\left(\log \delta^{-1}\right)^2, & \beta = \gamma,\\
  & c_4\delta^{-2-(\gamma-\beta)/\alpha}, & 0 < \beta < \gamma.
  \end{aligned}\right.
  \end{aligned}
\end{equation}
\end{proof}

\subsection{Proof of \cref{thm:costmc}, cost of MC}
\label{sec:appendix:costanalysis:compleixtyMC}

\begin{proof}
We can perform similar analysis for the adaptive MC given in \cref{alg:mc}. The difference is that, for MLMC the bound for the total variance in \eqref{eqn:mlmcvariance} is needed, while for MC, we only need the bound for the variance on each level, as $\Var\left[\hat{P}_l^{(l)}\right]$ in \eqref{eqn:mcvariance}.

We have similar estimate for the bias by taking $L = \left\lceil \frac{\log(\sqrt{2}c_1T^\alpha\delta^{-1})}{\alpha\log \Gamma}\right\rceil$,
  \begin{equation}
    \frac{1}{\sqrt{2}}\Gamma^{-\alpha}\delta < c_1 (\Delta t_L)^\alpha \leq \frac{1}{\sqrt{2}} \delta,
  \end{equation}
  where $\Gamma = \Delta t_{l-1}/\Delta t_l$, and equivalently,
\begin{equation}
  \left|\E\left[\hat{P}_{L}^{(L)} - P\right]\right|^2 = \left|\E\left[P^{(L)} - P\right]\right|^2 \leq \frac{1}{2} \delta^2.
\end{equation}
The variances of each level should be bounded by \begin{equation}
  \Var\left[\hat{P}_l^{(l)}\right] = N_l^{-1}\mathsf{V}_l \leq \frac{1}{2}\delta^2, \quad \Rightarrow \quad N_l = \left\lceil 2\delta^{-2} \mathsf{V}_{l}\right\rceil, \quad l=0,\dots,L.
  \label{eqn:mcvariance}
\end{equation}
to make sure that $\mse<\delta^2$.
The total computational cost $\mathfrak{C}^*$ is
\begin{equation}
  \mathfrak{C}^* = \sum_{l=0}^L N_l \mathfrak{C}\left[P_l^{(l)}\right] \leq 2\delta^{-2}  \sum_{l=0}^L \mathfrak{C}\left[P_l^{(l)}\right]\mathsf{V}_{l} + \sum_{l=0}^L \mathfrak{C}\left[P_l^{(l)}\right].
\end{equation}

With properties \textit{\textbf{A}}\ref{Assum:2-2}, \textit{\textbf{A}}\ref{Assum:2-3}, we have
\begin{equation}
  \sum_{l=0}^L \mathfrak{C}\left[P_l^{(l)}\right]\mathsf{V}_l \leq c_2c_3 \sum_{l=0}^L (\Delta t_l)^{\beta_0-\gamma} = \left\{ \begin{aligned}  
  & c_5, & \beta_0 > \gamma, \\
  & L+1 \leq c_6 \log \delta^{-1}, & \beta_0 = \gamma,\\
  & c_7 (\Delta t_L)^{\beta_0-\gamma} \leq c_8 \delta^{-(\gamma-\beta_0)/\alpha}, & 0 < \beta_0 < \gamma,
  \end{aligned}\right.
\end{equation}
and
\begin{equation}
  \sum_{l=0}^L \mathfrak{C}\left[P_l^{(l)}\right] \leq c_3 \sum_{l=0}^L (\Delta t_l)^{-\gamma} \leq c_9 \delta^{-\gamma/\alpha} \leq c_9\delta^{-2}.
\end{equation}

In summary, the computational cost $\mathfrak{C}^*$ is bounded by
\begin{equation}
  \begin{aligned}
  \mathfrak{C}^* &\leq \left\{ \begin{aligned}  
  & c_4\delta^{-2}, & \beta_0 > \gamma, \\
  & c_4\delta^{-2}\log \delta^{-1}, & \beta_0 = \gamma,\\
  & c_4\delta^{-2-(\gamma-\beta_0)/\alpha}, & 0 < \beta_0 < \gamma.
  \end{aligned}\right.
  \end{aligned}
\end{equation}
\end{proof}

\section{AP property for the deterministic method of the Jin-Xin model}
\label{sec:appendix:numericallimit}
For completeness, we formulate the AP property of the deterministic numerical method in \S~\ref{sec:comparison:randomchoice:apmc} in the following proposition.
\begin{proposition}
 The deterministic numerical method in \S~\ref{sec:comparison:randomchoice:apmc} is asymptotic preserving since it has the correct numerical $\varepsilon \to 0$ limit, namely, in the limit $\varepsilon \rightarrow 0$, the numerical solutions $u_i^n$, $v_i^n$ tends to the numerical solutions of the limit equations
 \begin{equation}\label{eqn:limitequation}
     v = bu, \quad  \frac{\partial u}{\partial t} + b\frac{\partial u}{\partial x} =0.
 \end{equation}
\end{proposition}

We provide a heuristic argument in the following, for complete theory, please refer to \cite{jin1999efficient}.
As $\varepsilon \rightarrow 0$, the implicit Euler's method used in relaxation step collapses to
\begin{equation} \label{eqn:numericalrelaxationlimit}
    v_i^{n+1} = bu_{i}^{n+1}, \quad n = 0, \dots, M_{l}-1.
\end{equation}
Applying \eqref{eqn:numericalrelaxationlimit} into the convection step gives
\begin{equation}
    \rho_{i}^{n} = (\sqrt{a}+b)u_{i}^{n}, \quad l_{i}^{n} = (\sqrt{a}-b)u_{i}^{n},
\end{equation}
and 
\begin{equation}
    u_{i}^{n+1} = (1-\nu) u_{i}^{n} + \frac{(\sqrt{a}+b)\nu}{2\sqrt{a}} u_{i-1}^{n} + \frac{(\sqrt{a}-b)\nu}{2\sqrt{a}} u_{i+1}^{n},
\end{equation}
which is equivalent to
\begin{equation}\label{eqn:numericallimit}
    \frac{u_{i}^{n+1}-u_{i}^{n}}{\Delta t} + b\frac{u_{i+1}^{n} - u_{i-1}^{n}}{2\Delta x} = \frac{\sqrt{a}\Delta x}{2} \frac{u_{i+1}^{n} - 2u_{i}^{n} + u_{i-1}^{n}}{(\Delta x)^2}.
\end{equation}
\eqref{eqn:numericalrelaxationlimit}, \eqref{eqn:numericallimit} is a first-order consistent and stable discretization of \eqref{eqn:limitequation}, with leading truncation error
\begin{equation}
    \frac{\sqrt{a}\Delta x}{2} \frac{\partial^2 u}{\partial x^2}.
\end{equation}
Thus this scheme has the correct asymptotic limit and is expected to capture the correct limiting behavior in the regime when $\varepsilon \ll 1$ and $\Delta t \gg \varepsilon$, with a first-order accuracy.

\section{Analysis for the APMC method in \S~\ref{sec:comparison:randomchoice:analysis}}
\label{sec:appendix:varianceanalysis}

\subsection{Proof of \cref{thm:det-energy}}
\begin{proof}
  Combining \eqref{con:char-convex} and \eqref{con:char-2uv}, we obtain the formula of $\E\left[u_i^{n+\frac{1}{2}}\right]$, $\E\left[v_i^{n+\frac{1}{2}}\right]$,
  \begin{equation}\label{det:uv-upw}
      \begin{aligned}
      \E\left[v_i^{n+\frac{1}{2}}\right] &= (1-\nu)\E\left[v_i^n\right]+\frac{\nu}{2} \left(\E\left[v_{i+1}^n\right]+\E\left[v_{i-1}^n\right]\right)-\frac{\nu\sqrt{a}}{2}\left(\E\left[u_{i+1}^n\right]-\E\left[u_{i-1}^n\right]\right),\\
      \E\left[u_i^{n+\frac{1}{2}}\right] &= (1-\nu)\E\left[u_i^n\right]+\frac{\nu}{2} \left(\E\left[u_{i+1}^n\right]+\E\left[u_{i-1}^n\right]\right)-\frac{\nu}{2\sqrt{a}}\left(\E\left[v_{i+1}^n\right]-\E\left[v_{i-1}^n\right]\right).\\
      \end{aligned}
  \end{equation}
  
  We compute the following quadratic sums and simplify them as follows
  \begin{equation}
    \begin{aligned}
      &\quad \sum_i\left[a\left(\E\left[u_i^{n+\frac{1}{2}}\right]\right)^2+\left(\E\left[v_i^{n+\frac{1}{2}}\right]\right)^2\right] = \frac{1}{2}\sum_i\left[\left(\E\left[\rho_i^{n+\frac{1}{2}}\right]\right)^2+\left(\E\left[l_i^{n+\frac{1}{2}}\right]\right)^2\right]\\
      &=\frac{(1-\nu)^2+\nu^2}{2}\sum_i\left[\left(\E\left[\rho_i^{n}\right]\right)^2+\left(\E\left[l_i^{n}\right]\right)^2\right]+\nu(1-\nu) \sum_i \left[\E\left[\rho_i^n\right]\E\left[\rho_{i+1}^n\right]+\E\left[l_i^n\right] \E\left[l_{i+1}^n\right]\right] \\
      &=\left[(1-\nu)^2+\nu^2\right]\sum_i\left[a\left(\E\left[u_i^{n}\right]\right)^2+\left(\E\left[v_i^{n}\right]\right)^2\right]+2\nu(1-\nu) \sum_i \left[a\E\left[u_i^n\right]\E\left[u_{i+1}^n\right]+\E\left[v_i^n\right] \E\left[v_{i+1}^n\right]\right], \\
    \end{aligned}
  \end{equation}
  \begin{equation}\label{det:square-uv}
    \begin{aligned}
      &\quad \sum_i\left[\E\left[u_i^{n+\frac{1}{2}}\right]\E\left[v_i^{n+\frac{1}{2}}\right]\right] = \frac{1}{4\sqrt{a}}\sum_i\left[\left(\E\left[\rho_i^{n+\frac{1}{2}}\right]\right)^2-\left(\E\left[l_i^{n+\frac{1}{2}}\right]\right)^2\right]\\
      &=\frac{(1-\nu)^2+\nu^2}{4\sqrt{a}}\sum_i\left[\left(\E\left[\rho_i^{n}\right]\right)^2-\left(\E\left[l_i^{n}\right]\right)^2\right]+\frac{\nu(1-\nu)}{2\sqrt{a}} \sum_i \left[\E\left[\rho_i^n\right]\E\left[\rho_{i+1}^n\right]-\E\left[l_i^n\right] \E\left[l_{i+1}^n\right]\right] \\
      &=\left[(1-\nu)^2+\nu^2\right]\sum_i\left[\E\left[u_i^{n}\right]\E\left[v_i^{n}\right]\right]+\nu(1-\nu) \sum_i \left[\E\left[u_i^n\right]\E\left[v_{i+1}^n\right]+\E\left[v_i^n\right] \E\left[u_{i+1}^n\right]\right], \\
    \end{aligned}
  \end{equation}
  which lead to the following estimate
  \begin{equation}\label{det:energy-h-step}
    \begin{aligned}
      \mathcal{E}_{n+\frac{1}{2}}&= \sum_i\left[(a-b^2)\left(\E\left[u_i^{n+\frac{1}{2}}\right]\right)^2+\left(\E\left[bu_i^{n+\frac{1}{2}} - v_i^{n+\frac{1}{2}}\right]\right)^2\right]\Delta x \\
      &= \sum_i\left[a\left(\E\left[u_i^{n+\frac{1}{2}}\right]\right)^2+\left(\E\left[v_i^{n+\frac{1}{2}}\right]\right)^2-2b\E\left[u_i^{n+\frac{1}{2}}\right]\E\left[v_i^{n+\frac{1}{2}}\right]\right]\Delta x  \\
      &= [(1-\nu)^2+\nu^2]\sum_i\left[(a-b^2)\left(\E\left[u_i^{n}\right]\right)^2+\left(\E\left[bu_i^{n} - v_i^{n}\right]\right)^2\right]\Delta x \\
      &\quad + 2\nu(1-\nu)\sum_i\left[(a-b^2)\left(\E\left[u_i^{n}\right]\E\left[u_{i+1}^n\right]\right) + \E\left[bu_i^{n} - v_i^{n}\right]\E\left[bu_{i+1}^{n}-v_{i+1}^{n}\right]\right]\Delta x \\
      &\leq \sum_i\left[(a-b^2)\left(\E\left[u_i^{n}\right]\right)^2+\left(\E\left[bu_i^{n} - v_i^{n}\right]\right)^2\right]\Delta x =\mathcal{E}_n.\\
    \end{aligned}
  \end{equation}
  
  In the relaxation step, by \eqref{rel:convex}, we have
  \begin{equation}\label{det:rel-uv}
    \E\left[u_i^{n+1}\right] = \E\left[u_i^{n+\frac{1}{2}}\right], \quad \E\left[bu_i^{n+1}-v_i^{n+1}\right] = \displaystyle p \E\left[bu_i^{n+\frac{1}{2}} - v_i^{n+\frac{1}{2}}\right],
  \end{equation}
  thus
  \begin{equation}\label{det:energy-1-step}
    \begin{aligned}
      \mathcal{E}_{n+1} &= \sum_i\left[(a-b^2)\left(\E\left[u_i^{n+\frac{1}{2}}\right]\right)^2+\displaystyle p^2\left(\E\left[bu_i^{n+\frac{1}{2}} - v_i^{n+\frac{1}{2}}\right]\right)^2\right]\Delta x \leq \mathcal{E}_{n+\frac{1}{2}}\leq \mathcal{E}_n.
    \end{aligned}
  \end{equation}
  Therefore the energy decreases monotonically. And \eqref{det:prop-ubound} can be easily derived from \eqref{det:prop-enerdecay}. The boundedness of $\left\|\E\left[v^n\right]\right\|_{2}$ \eqref{det:prop-vbound} follows by the fact that 
  \begin{equation}\label{det:vtobuv}
    \sum_i\left(\E\left[v_i^n\right]\right)^2\Delta x \leq 2 \sum_i\left[\left(\E\left[bu_i^n - v_i^n\right]\right)^2 + b^2\left(\E\left[u_i^n\right]\right)^2\right]\Delta x \leq \frac{2a}{a-b^2}\mathcal{E}_n.
  \end{equation}

\end{proof}

\subsection{Proof of \cref{prop:det-buv}}
\begin{proof}
  \eqref{det:energy-1-step} implies that
  \begin{equation}\label{det:energy-1-step-plus}
    \begin{aligned}
      \mathcal{E}_{n+1} +\sum_i(1-\displaystyle p^2)\left(\E\left[bu_i^{n+\frac{1}{2}} - v_i^{n+\frac{1}{2}}\right]\right)^2\Delta x = \mathcal{E}_{n+\frac{1}{2}}\leq \mathcal{E}_n.
    \end{aligned}
  \end{equation}
  Therefore
  \begin{equation}\label{det:energy-full}
  \begin{aligned}
    &\quad\mathcal{E}_{n+1} +\sum_{j=0}^n\sum_i(1-\displaystyle p^2)\left(\E\left[bu_i^{j+\frac{1}{2}}-v_i^{j+\frac{1}{2}}\right]\right)^2\Delta x\\
    &\leq \mathcal{E}_{n} +\sum_{j=0}^{n-1}\sum_i(1-\displaystyle p^2)\left(\E\left[bu_i^{j+\frac{1}{2}}-v_i^{j+\frac{1}{2}}\right]\right)^2\Delta x\\
    &\leq \cdots\\
    &\leq \mathcal{E}_{1} +\sum_i(1-\displaystyle p^2)\left(\E\left[bu_i^{\frac{1}{2}}-v_i^{\frac{1}{2}}\right]\right)^2\Delta x\\
    &\leq \mathcal{E}_0.\\
  \end{aligned}
\end{equation}
Noting that $\E\left[bu_i^{n+1}-v_i^{n+1}\right] = \displaystyle p\E\left[bu_i^{n+\frac{1}{2}} - v_i^{n+\frac{1}{2}}\right]$, we have
\begin{equation}\label{det:sum-square-buv}
  \sum_{j=1}^{n+1} \sum_i \left(\E\left[bu_i^{j}-v_i^{j}\right]\right)^2 \Delta x = \displaystyle p^2\sum_{j=0}^n\sum_i\left(\E\left[bu_i^{j+\frac{1}{2}}-v_i^{j+\frac{1}{2}}\right]\right)^2\Delta x\leq \frac{\displaystyle p^2}{1-\displaystyle p^2} \mathcal{E}_0,\quad \forall n,
\end{equation}
namely, \eqref{det:prop-sum-vbu} holds. Therefore, as $ \varepsilon\rightarrow 0$, $\displaystyle p\rightarrow 0$, $\E\left[bu^n-v^n\right]$ converges to zero in $\ell^2$.
\end{proof}

\subsection{Proof of \cref{thm:s-sto-energy}, for the semi-random APMC}
\begin{proof}
We have the following identities, for the change of variances in the convection step,
\begin{equation}\label{s-sto:var-uv}
  \begin{aligned}
    &\quad \sum_i\left[a\Var\left[\tilde{u}_i^{n+\frac{1}{2}}\right]+\Var\left[\tilde{v}_i^{n+\frac{1}{2}}\right]\right] = \frac{1}{2}\sum_i\left[\Var\left[\tilde{\rho}_i^{n+\frac{1}{2}}\right]+\Var\left[\tilde{l}_i^{n+\frac{1}{2}}\right]\right]\\
    &=\left[(1-\nu)^2+\nu^2\right]\sum_i\left[a\Var\left[\tilde{u}_i^{n}\right]+\Var\left[\tilde{v}_i^{n}\right]\right]+2\nu(1-\nu) \sum_i a\Cov\left[\tilde{u}_i^n,\tilde{u}_{i+1}^n\right]+\Cov\left[\tilde{v}_i^n, \tilde{v}_{i+1}^n\right], \\
    &\quad \sum_i\Cov\left[\tilde{u}_i^{n+\frac{1}{2}},\tilde{v}_i^{n+\frac{1}{2}}\right] = \frac{1}{4\sqrt{a}}\sum_i\left[\Var\left[\tilde{\rho}_i^{n+\frac{1}{2}}\right]-\Var\left[\tilde{l}_i^{n+\frac{1}{2}}\right]\right]\\
    &=\left[(1-\nu)^2+\nu^2\right]\sum_i\Cov\left[\tilde{u}_i^{n},\tilde{v}_i^{n}\right]+\nu(1-\nu) \sum_i \left[\Cov\left[\tilde{u}_i^n,\tilde{v}_{i+1}^n\right]+\Cov\left[\tilde{v}_i^n,\tilde{u}_{i+1}^n\right]\right]. \\
  \end{aligned}
\end{equation}

Combining above identities, we have
\begin{equation}\label{s-sto:energy-h-step}
  \begin{aligned}
    &\quad \sum_i\left[(a-b^2)\Var\left[\tilde{u}_i^{n+\frac{1}{2}}\right]+\Var\left[b\tilde{u}_i^{n+\frac{1}{2}}-\tilde{v}_i^{n+\frac{1}{2}}\right]\right] \\
    &= [(1-\nu)^2+\nu^2]\sum_i\left[(a-b^2)\Var\left[\tilde{u}_i^{n}\right]+\Var\left[b\tilde{u}_i^{n}-\tilde{v}_i^{n}\right]\right]\\
    &\quad + 2\nu(1-\nu)\sum_i\left[(a-b^2)\Cov\left[\tilde{u}_i^{n},\tilde{u}_{i+1}^n\right] + \Cov\left[b\tilde{u}_i^{n}-\tilde{v}_i^{n},b\tilde{u}_{i+1}^{n}-\tilde{v}_{i+1}^{n}\right]\right]\\
    &\leq \sum_i\left[(a-b^2)\Var\left[\tilde{u}_i^{n}\right]+\Var\left[b\tilde{u}_i^{n}-\tilde{v}_i^{n}\right]\right].\\
  \end{aligned}
\end{equation}

We then treat the relaxation step, using the conditional variance formula
\begin{equation}\label{s-sto:cond-var-v}
  \begin{aligned}
    \Var\left[\tilde{u}_i^{n+1}\right] & = \Var\left[\tilde{u}_i^{n+\frac{1}{2}}\right],\\
    \Var\left[b\tilde{u}_i^{n+1}-\tilde{v}_i^{n+1}\right] &= \displaystyle p \Var\left[b\tilde{u}_i^{n+\frac{1}{2}}-\tilde{v}_i^{n+\frac{1}{2}}\right]+\displaystyle p(1-\displaystyle p)\left(\E\left[b\tilde{u}_i^{n+\frac{1}{2}}-\tilde{v}_i^{n+\frac{1}{2}}\right]\right)^2.\\
  \end{aligned}
\end{equation}

By substituting \eqref{s-sto:energy-h-step} into \eqref{s-sto:cond-var-v}, we have 
\begin{equation}\label{s-sto:energy-1-step}
  \begin{aligned}
    &\quad \sum_i\left[(a-b^2)\Var\left[\tilde{u}_i^{n+1}\right]+\Var\left[b\tilde{u}_i^{n+1}-\tilde{v}_i^{n+1}\right]\right] \\
    &\leq \sum_i\left[(a-b^2)\Var\left[\tilde{u}_i^{n}\right]+\Var\left[b\tilde{u}_i^{n}-\tilde{v}_i^{n}\right]\right]+\displaystyle p(1-\displaystyle p)\sum_i\left(\E\left[b\tilde{u}_i^{n+\frac{1}{2}}-\tilde{v}_i^{n+\frac{1}{2}}\right]\right)^2\\
    &\leq \cdots\\
    &\leq 0 + \displaystyle p(1-\displaystyle p)\sum_{j=0}^n \sum_i \left(\E\left[b\tilde{u}_i^{j+\frac{1}{2}}-\tilde{v}_i^{j+\frac{1}{2}}\right]\right)^2\\
    &= \frac{1-\displaystyle p}{\displaystyle p}\sum_{j=0}^n \sum_i \left(\E\left[b\tilde{u}_i^{j+1}-\tilde{v}_i^{j+1}\right]\right)^2.\\
  \end{aligned}
\end{equation}
\eqref{s-sto:prop-var-bd} follows by a combination of \eqref{s-sto:energy-1-step} and \cref{prop:det-buv}.
\end{proof}

\subsection{Proof of \cref{thm:f-sto-energy}, for the fully-random APMC}
\begin{proof}
We have the similar identities for the  change of the variances, when we compute $u_i^{n+\frac{1}{2}},v_i^{n+\frac{1}{2}}$ from the characteristic variables $\rho_i^{n+\frac{1}{2}}$, $l_i^{n+\frac{1}{2}}$,
\begin{equation}\label{f-sto:var-auv}
  \begin{aligned}
    \sum_i \left[a\Var\left[u_i^{n+\frac{1}{2}}\right]+\Var\left[v_i^{n+\frac{1}{2}}\right]\right] &= \sum_i \left[\Var\left[\frac{\rho_i^{n+\frac{1}{2}}+l_i^{n+\frac{1}{2}}}{2}\right]+\Var\left[\frac{\rho_i^{n+\frac{1}{2}}-l_i^{n+\frac{1}{2}}}{2}\right]\right] \\ &=\frac{1}{2}\sum_i\left[\Var\left[\rho_i^{n+\frac{1}{2}}\right]+\Var\left[l_i^{n+\frac{1}{2}}\right]\right],\\
    \sum_i\Cov\left[u_i^{n+\frac{1}{2}},v_i^{n+\frac{1}{2}}\right] &= \sum_i\Cov\left[\frac{\rho_i^{n+\frac{1}{2}}+l_i^{n+\frac{1}{2}}}{2\sqrt{a}},\frac{\rho_i^{n+\frac{1}{2}}-l_i^{n+\frac{1}{2}}}{2}\right] \\
    &=\frac{1}{4\sqrt{a}}\sum_i\left[\Var\left[\rho_i^{n+\frac{1}{2}}\right]-\Var\left[l_i^{n+\frac{1}{2}}\right]\right].\\
  \end{aligned}
\end{equation}

Combining the identities in \eqref{f-sto:var-auv}, we have
\begin{equation}\label{f-sto:energy-hh-step}
  \begin{aligned}
    &\quad \sum_i\left[(a-b^2)\Var\left[u_i^{n+\frac{1}{2}}\right]+\Var\left[bu_i^{n+\frac{1}{2}}-v_i^{n+\frac{1}{2}}\right]\right] \\
    &=\frac{1}{2}\sum_i\left[\Var\left[\rho_i^{n+\frac{1}{2}}\right]+\Var\left[l_i^{n+\frac{1}{2}}\right]\right]-\frac{b}{2\sqrt{a}}\sum_i\left[\Var\left[\rho_i^{n+\frac{1}{2}}\right]-\Var\left[l_i^{n+\frac{1}{2}}\right]\right]\\
    &=\frac{1}{2}\sum_i\left[\left(1-\frac{b}{\sqrt{a}}\right)\Var\left[\rho_i^{n+\frac{1}{2}}\right]+\left(1+\frac{b}{\sqrt{a}}\right)\Var\left[l_i^{n+\frac{1}{2}}\right]\right].\\
  \end{aligned}
\end{equation}

We apply the conditional variance formula
\begin{equation}\label{f-sto:con-var-rho}
  \begin{aligned}
    &\Var\left[\rho_i^{n+\frac{1}{2}}\right]=(1-\nu)\Var\left[\rho_i^{n}\right]+\nu\Var\left[\rho_{i-1}^{n}\right]+\nu(1-\nu)\left(\E\left[\rho_{i}^n-\rho_{i-1}^n\right]\right)^2,\\
    &\Var\left[l_i^{n+\frac{1}{2}}\right]=(1-\nu)\Var\left[l_i^{n}\right]+\nu\Var\left[l_{i+1}^{n}\right]+\nu(1-\nu)\left(\E\left[l_{i+1}^n-l_{i}^n\right]\right)^2,\\
  \end{aligned}
\end{equation}
and write
\begin{equation}\label{f-sto:var-rho2u}
  \begin{aligned}
    &\Var\left[\rho_i^{n}\right]=\Var\left[\sqrt{a}u_i^{n}\right]+\Var\left[v_i^{n}\right]+2\sqrt{a}\Cov\left[u_{i}^n,v_{i}^n\right],\\
    &\Var\left[l_i^{n}\right]=\Var\left[\sqrt{a}u_i^{n}\right]+\Var\left[v_i^{n}\right]-2\sqrt{a}\Cov\left[u_{i}^n,v_{i}^n\right].\\
  \end{aligned}
\end{equation}

By substituting \eqref{f-sto:con-var-rho}, \eqref{f-sto:var-rho2u} into \eqref{f-sto:energy-hh-step}, it leads to the following identity
\begin{equation}\label{f-sto:energy-h-s}
  \begin{aligned}
    &\quad \sum_i\left[(a-b^2)\Var\left[u_i^{n+\frac{1}{2}}\right]+\Var\left[bu_i^{n+\frac{1}{2}}-v_i^{n+\frac{1}{2}}\right]\right] \\
    &=\sum_i\left[a\Var\left[u_i^{n}\right]+\Var\left[v_i^{n}\right]-2b\Cov\left[u_i^{n},v_i^{n}\right]\right] \\
    &\quad +\frac{\nu(1-\nu)}{2}\left[\left(1-\frac{b}{\sqrt{a}}\right)\sum_i\left(\E\left[\rho_i^{n}-\rho_{i-1}^{n}\right]\right)^2+\left(1+\frac{b}{\sqrt{a}}\right)\sum_i\left(\E\left[l_{i+1}^{n}-l_{i}^{n}\right]\right)^2\right].
  \end{aligned}
\end{equation}

We then derive
\begin{equation}\label{f-sto:varia-extra}
  \begin{aligned}
    &\quad \frac{1}{2}\left[\left(1-\frac{b}{\sqrt{a}}\right)\sum_i\left(\E\left[\rho_i^{n}-\rho_{i-1}^{n}\right]\right)^2+\left(1+\frac{b}{\sqrt{a}}\right)\sum_i\left(\E\left[l_{i+1}^{n}-l_{i}^{n}\right]\right)^2\right]\\
    &= \sum_i\left[a\left(\E\left[u_{i+1}^{n}-u_{i}^{n}\right]\right)^2+\left(\E\left[v_{i+1}^{n}-v_{i}^{n}\right]\right)^2-2b\E\left[u_{i+1}^{n}-u_{i}^{n}\right]\E\left[v_{i+1}^{n}-v_{i}^{n}\right]\right]\\
    &= \sum_i\left[(a-b^2)\left(\E\left[u_{i+1}^{n}-u_{i}^{n}\right]\right)^2+\left(\E\left[b\left(u_{i+1}^{n}-u_i^{n+1}\right)-\left(v_{i+1}^{n}-v_i^{n+1}\right)\right]\right)^2\right]:=V_n, \\
  \end{aligned}
\end{equation}
by noticing that
\begin{equation}\label{f-sto:varia-rho}
  \begin{aligned}
    \left(\E\left[\rho_i^{n}-\rho_{i-1}^{n}\right]\right)^2 &= \left(\E\left[\sqrt{a}u_i^{n}-\sqrt{a}u_{i-1}^{n}+v_i^{n}-v_{i-1}^{n}\right]\right)^2,\\
    \left(\E\left[l_{i+1}^{n}-l_{i}^{n}\right]\right)^2 &= \left(\E\left[\sqrt{a}u_{i+1}^{n}-\sqrt{a}u_{i}^{n}-v_{i+1}^{n}-v_{i}^{n}\right]\right)^2.
  \end{aligned}
\end{equation}

In the convective part, we have
\begin{equation}\label{f-sto:energy-h-step}
\begin{aligned}
    &\quad \sum_i\left[(a-b^2)\Var\left[u_i^{n+\frac{1}{2}}\right]+\Var\left[bu_i^{n+\frac{1}{2}}-v_i^{n+\frac{1}{2}}\right]\right] \\
    &=\sum_i\left[(a-b^2)\Var\left[u_i^{n}\right]+\Var\left[bu_i^{n}-v_i^{n}\right]\right] +\nu(1-\nu)V_n,
\end{aligned}
 \end{equation}
and also $V_n \Delta x\leq \mathcal{C}_0(\Delta x)^2$ by \eqref{rem:square-variation}.

The analysis of the relaxation part is similar to that of \eqref{s-sto:cond-var-v}, namely,
\begin{equation}\label{f-sto:energy-full}
  \begin{aligned}
    &\quad \sum_i\left[(a-b^2)\Var\left[u_i^{n+1}\right]+\Var\left[bu_i^{n+1}-v_i^{n+1}\right]\right] \\
    &\leq \sum_i\left[(a-b^2)\Var\left[u_i^{n+\frac{1}{2}}\right]+\Var\left[bu_i^{n+\frac{1}{2}}-v_i^{n+\frac{1}{2}}\right]\right]+\displaystyle p(1-\displaystyle p)\sum_i\left(\E\left[bu_i^{n+\frac{1}{2}}-v_i^{n+\frac{1}{2}}\right]\right)^2 \\
    &\leq \sum_i\left[(a-b^2)\Var\left[u_i^{n}\right]+\Var\left[bu_i^{n}-v_i^{n}\right]\right]+\displaystyle p(1-\displaystyle p)\sum_i\left(\E\left[bu_i^{n+\frac{1}{2}}-v_i^{n+\frac{1}{2}}\right]\right)^2 + \nu(1-\nu)V_n \\
    &\leq \cdots\\
    &\leq \frac{(1-\displaystyle p)}{\displaystyle p}\sum_{j=1}^{n+1} \sum_i\left(\E\left[bu_i^{j}-v_i^{j}\right]\right)^2+\nu(1-\nu)\sum_{j=0}^nV_n.\\
  \end{aligned}
\end{equation}
Thus
\begin{equation}\label{f-sto:energy-bound}
  \begin{aligned}
    \sum_i \left[(a-b^2)\Var\left[u_i^{n}\right]+\Var\left[bu_i^{n}-v_i^{n}\right]\right] \Delta x &\leq \frac{\displaystyle p}{1+\displaystyle p} \mathcal{E}_0+\nu(1-\nu)n\mathcal{C}_0 (\Delta x)^2\\
    &\leq \frac{\displaystyle p}{1+\displaystyle p} \mathcal{E}_0+\frac{a(1-\nu)}{\nu}T\mathcal{C}_0 \Delta t, \quad \forall n,
  \end{aligned}
\end{equation}
where
\begin{equation*}
  \begin{aligned}
    \mathcal{E}_0 & = \sum_i\left[(a-b^2)\left(u_i^{0}\right)^2+\left(bu_i^{0}-v_i^{0}\right)^2\right] \Delta x,\\
    \mathcal{C}_0 & = \sum_i\left[(a-b^2)\left(\frac{u_{i+1}^0-u_i^{0}}{\Delta x}\right)^2+\left[b\left(\frac{u_{i+1}^0-u_i^{0}}{\Delta x}\right)-\left(\frac{v_{i+1}^0-v_i^{0}}{\Delta x}\right)\right]^2\right] \Delta x.
  \end{aligned}
\end{equation*}
\end{proof}

\end{document}

%% file: ex_shared.tex
\usepackage[a4paper, total={6in, 9in}]{geometry}
\usepackage{lipsum}
\usepackage{amssymb}
\usepackage{amsfonts}
\usepackage{graphicx}
\usepackage{epstopdf}
\usepackage{algorithmic}

\usepackage{multirow}
\usepackage{enumerate}
\usepackage{subfigure}
\usepackage{dsfont}
\usepackage{bm}
\usepackage{hyperref}
\usepackage{upgreek}
\usepackage{float}
\usepackage{url}
\usepackage{color}
\usepackage[normalem]{ulem}
\usepackage{soul}

\def\E{\mathrm{E}}
\def\Var{\mathrm{Var}}
\def\Cov{\mathrm{Cov}}
\def\mse{\mathrm{MSE}}

\ifpdf
  \DeclareGraphicsExtensions{.eps,.pdf,.png,.jpg}
\else
  \DeclareGraphicsExtensions{.eps}
\fi

\makeatletter
\newcommand*\bigcdot{\mathpalette\bigcdot@{.5}}
\newcommand*\bigcdot@[2]{\mathbin{\vcenter{\hbox{\scalebox{#2}{$\m@th#1\bullet$}}}}}
\makeatother

\newsiamremark{remark}{Remark}
\newsiamremark{hypothesis}{Hypothesis}
\crefname{hypothesis}{Hypothesis}{Hypotheses}
\newsiamthm{claim}{Claim}

\headers{On MLMC for hyperbolic equations}{Junpeng Hu, Shi Jin, Jinglai Li and Lei Zhang}

\title{On multilevel Monte Carlo methods for deterministic and uncertain hyperbolic systems\thanks{Submitted to the editors DATE.
\funding{SJ was supported by  National Key R\&D Program of China Project No.
2020YFA0712000, the Strategic Priority Research Program of Chinese Academy of Sciences,
XDA25010404 and NSFC grant No. 12031013. LZ's work was partially supported by NSFC grants No. 11871339 and 11861131004.}}}

\author{Junpeng Hu\thanks{School of Mathematical Sciences, Institute of Natural Sciences, and MOE-LSC, Shanghai Jiao Tong University, Shanghai, 200240, China
  (\email{hjp3268@sjtu.edu.cn}, \email{shijin-m@sjtu.edu.cn}, \email{lzhang2012@sjtu.edu.cn}).}
\and Shi Jin\footnotemark[2]
\and Jinglai Li\thanks{School of Mathematics, University of Birmingham, Birmingham, B15 2TT, United Kingdom
  (\email{j.li.10@bham.ac.uk})}
\and Lei Zhang\footnotemark[2]}

\usepackage{amsopn}


%% file: ex_article.bbl
\begin{thebibliography}{10}

\bibitem{abdulle2013stabilized}
{\sc A.~Abdulle and A.~Blumenthal}, {\em Stabilized multilevel monte carlo
  method for stiff stochastic differential equations}, Journal of Computational
  Physics, 251 (2013), pp.~445--460.

\bibitem{AbMi}
{\sc R.~Abgrall and S.~Mishra}, {\em Uncertainty quantification for hyperbolic
  systems of conservation laws}, in Handbook of Numerical Analysis, vol.~18,
  Elsevier, 2017, pp.~507--544.

\bibitem{ayi2019analysis}
{\sc N.~Ayi and E.~Faou}, {\em Analysis of an asymptotic preserving scheme for
  stochastic linear kinetic equations in the diffusion limit}, SIAM/ASA Journal
  on Uncertainty Quantification, 7 (2019), pp.~760--785.

\bibitem{BaoJin}
{\sc W.~Bao and S.~Jin}, {\em The random projection method for hyperbolic
  conservation laws with stiff reaction terms}, Journal of Computational
  Physics, 163 (2000), pp.~216--248.

\bibitem{barth2011multi}
{\sc A.~Barth, C.~Schwab, and N.~Zollinger}, {\em {Multi-level Monte Carlo
  finite element method for elliptic PDEs with stochastic coefficients}},
  Numerische Mathematik, 119 (2011), pp.~123--161.

\bibitem{chorin1976}
{\sc A.~J. Chorin}, {\em Random choice solution of hyperbolic systems}, Journal
  of computational physics, 22 (1976), pp.~517--533.

\bibitem{chorin-RCM}
{\sc A.~J. Chorin}, {\em Random choice methods with applications to reacting
  gas flow}, Journal of computational physics, 25 (1977), pp.~253--272.

\bibitem{cliffe2011multilevel}
{\sc K.~A. Cliffe, M.~B. Giles, R.~Scheichl, and A.~L. Teckentrup}, {\em
  {Multilevel Monte Carlo methods and applications to elliptic PDEs with random
  coefficients}}, Computing and Visualization in Science, 14 (2011), pp.~3--15.

\bibitem{dereich2011multilevel}
{\sc S.~Dereich and F.~Heidenreich}, {\em A multilevel monte carlo algorithm
  for l{\'e}vy-driven stochastic differential equations}, Stochastic Processes
  and their Applications, 121 (2011), pp.~1565--1587.

\bibitem{Dim-Par}
{\sc G.~Dimarco and L.~Pareschi}, {\em Hybrid multiscale methods for hyperbolic
  problems i. hyperbolic relaxation problems}, Communications in Mathematical
  Sciences, 4 (2006), pp.~155--177.

\bibitem{dimarco2018asymptotic}
{\sc G.~Dimarco, L.~Pareschi, and G.~Samaey}, {\em {Asymptotic-Preserving Monte
  Carlo methods for transport equations in the diffusive limit}}, SIAM Journal
  on Scientific Computing, 40 (2018), pp.~A504--A528.

\bibitem{giles2008multilevel}
{\sc M.~B. Giles}, {\em {Multilevel Monte Carlo path simulation}}, Operations
  research, 56 (2008), pp.~607--617.

\bibitem{Glimm1965solutions}
{\sc J.~Glimm}, {\em Solutions in the large for nonlinear hyperbolic systems of
  equations}, Communications on Pure and Applied Mathematics, 18 (1965),
  pp.~697--715.

\bibitem{Holden1996}
{\sc H.~Holden, B.~{\O}ksendal, J.~Ub{\o}e, and T.~Zhang}, {\em Stochastic
  partial differential equations}, Birkh{\"a}user Boston, Boston, MA, 1996,
  pp.~141--191.

\bibitem{jin1999efficient}
{\sc S.~Jin}, {\em Efficient asymptotic-preserving ({AP}) schemes for some
  multiscale kinetic equations}, SIAM Journal on Scientific Computing, 21
  (1999), pp.~441--454.

\bibitem{JP-UQ}
{\sc S.~Jin and L.~Pareschi}, {\em Uncertainty quantification for hyperbolic
  and kinetic equations}, vol.~14, Springer, 2018.

\bibitem{jin2000uniformly}
{\sc S.~Jin, L.~Pareschi, and G.~Toscani}, {\em {Uniformly accurate diffusive
  relaxation schemes for multiscale transport equations}}, SIAM Journal on
  Numerical Analysis, 38 (2000), pp.~913--936.

\bibitem{jin1995relaxation}
{\sc S.~Jin and Z.~Xin}, {\em {The relaxation schemes for systems of
  conservation laws in arbitrary space dimensions}}, Communications on pure and
  applied mathematics, 48 (1995), pp.~235--276.

\bibitem{kappeli2014wellbalanced}
{\sc R.~K{\"a}ppeli and S.~Mishra}, {\em Well-balanced schemes for the euler
  equations with gravitation}, Journal of Computational Physics, 259 (2014),
  pp.~199--219.

\bibitem{Karhunen1947Uber}
{\sc K.~Karhunen}, {\em \"{U}ber lineare {M}ethoden in der
  {W}ahrscheinlichkeitsrechnung}, Ann. Acad. Sci. Fennicae Ser. A. I.
  Math.-Phys., 1947 (1947), p.~79.

\bibitem{kurganov2018finite}
{\sc A.~Kurganov}, {\em Finite-volume schemes for shallow-water equations},
  Acta Numerica, 27 (2018), pp.~289--351.

\bibitem{liu2010}
{\sc D.~Liu}, {\em Uncertainty quantification with shallow water equations},
  University of Florence,  (2010).

\bibitem{Loeve1946Fonctions}
{\sc M.~Lo\`eve}, {\em Fonctions al\'{e}atoires de second ordre}, Revue Sci.,
  84 (1946), pp.~195--206.

\bibitem{lovbak2021multilevel}
{\sc E.~L{\o}vbak, G.~Samaey, and S.~Vandewalle}, {\em {A multilevel Monte
  Carlo method for asymptotic-preserving particle schemes in the diffusive
  limit}}, Numerische Mathematik, 148 (2021), pp.~141--186.

\bibitem{mishra2012mlmcconservationlaw}
{\sc S.~Mishra, C.~Schwab, and J.~{\v S}ukys}, {\em Multi-level monte carlo
  finite volume methods for nonlinear systems of conservation laws in
  multi-dimensions}, Journal of Computational Physics, 231 (2012),
  pp.~3365--3388.

\bibitem{mishra2012mlmcshallow}
{\sc S.~Mishra, C.~Schwab, and J.~{\v S}ukys}, {\em Multilevel monte carlo
  finite volume methods for shallow water equations with uncertain topography
  in multi-dimensions}, SIAM Journal on Scientific Computing, 34 (2012),
  pp.~B761--B784.

\bibitem{najm-UQCFD}
{\sc H.~N. Najm}, {\em Uncertainty quantification and polynomial chaos
  techniques in computational fluid dynamics}, Annual review of fluid
  mechanics, 41 (2009), pp.~35--52.

\bibitem{Natalini2001recent}
{\sc R.~Natalini}, {\em Recent mathematical results on hyperbolic relaxation
  problems}, in Analysis of Systems of Conservation Laws, 1999, pp.~128--198.

\bibitem{tryoen2010}
{\sc J.~Tryoen, O.~Le~Maitre, M.~Ndjinga, and A.~Ern}, {\em Intrusive galerkin
  methods with upwinding for uncertain nonlinear hyperbolic systems}, Journal
  of Computational Physics, 229 (2010), pp.~6485--6511.

\bibitem{Walsh1986}
{\sc J.~B. Walsh}, {\em An introduction to stochastic partial differential
  equations}, in {\'E}cole d'{\'E}t{\'e} de Probabilit{\'e}s de Saint Flour XIV
  - 1984, P.~L. Hennequin, ed., Berlin, Heidelberg, 1986, Springer Berlin
  Heidelberg, pp.~265--439.

\end{thebibliography}
